\documentclass[11pt,reqno]{amsart}

 \usepackage{amsmath,amsthm,amsfonts}
\usepackage{latexsym,mathabx,pxfonts}
\usepackage{amssymb}

\newcommand{\om}{\omega}

\newcommand{\lam}{\lambda}
\newcommand{\al}{\alpha}

\newcommand{\fy}{\varphi}

\newcommand{\p}{\partial}
\newcommand{\I}{\infty}
\newcommand{\wt}[1]{\widetilde {#1}}
\newcommand{\ti}{\wt}
\renewcommand\tilde{\wt}
\newcommand{\R}{\mathbb{R}}

\newcommand{\Z}{\mathbb{Z}}

\renewcommand{\Im}{\mathop{\mathrm{Im}}}

\def\calL{{\mathcal L}}

\def\la{\lambda}

\numberwithin{equation}{section}

\newtheorem{thm}{Theorem}[section]
\newtheorem{cor}[thm]{Corollary}
\newtheorem{lem}[thm]{Lemma}
\newtheorem{prop}[thm]{Proposition}

\theoremstyle{remark}

\newtheorem{defi}{Definition}[section]

\newcommand{\ran}{\rangle}
\newcommand{\lan}{\langle}

\def\less{\lesssim}

\newcommand{\EQ}[1]{\begin{equation}  \begin{split} #1 \end{split} \end{equation} }
\newcommand{\Del}[1]{}

\newcommand{\pr}{\\ &}

\newcommand{\HH}{\mathcal{H}}

\def\pr{\partial}

\def\nn{\nonumber}

\def\calQ{{\mathcal Q}}

\def\eps{\varepsilon}

\def\const{\mathrm{const}}

\def\f{\frac}

\def\fy{\varphi}
\def\calR{\mathcal{R}}
\def\lan{\langle}
\def\ran{\rangle}
\def\const{\mathrm{const}}
\def\wt{\widetilde}
\newcommand{\LL}{\mathcal{L}}
\def\cL{\LL}
\def\cD{\mathcal{D}}
\renewcommand\ln{\log}
\def\ti{\tilde}
\def\cQ{\mathcal{Q}}

\newcommand{\IS}{\mbox{\,{\rm IS}}}
\def\calD{\mathcal{D}}

\begin{document}

\title[Full range of blow up exponents]{Full range of blow up exponents for the quintic wave equation in three dimensions}

\author{Joachim Krieger, Wilhelm Schlag}



\thanks{Support of the National Science Foundation  DMS-0617854, DMS-1160817 for the second author, and  the Swiss National Fund for
the first author  are gratefully acknowledged. The latter would like to thank the University of Chicago for its hospitality in August 2012}

\begin{abstract}
For the critical focusing wave equation $\Box u = u^5$ on $\R^{3+1}$ in the radial case, we prove the existence of type II blow up solutions with scaling parameter $\lambda(t) = t^{-1-\nu}$ for all $\nu>0$. This extends the previous work by the authors and Tataru  where the condition $\nu>\f12$ had been imposed, and gives the optimal range of polynomial blow up rates in light of recent work by Duyckaerts, Kenig and Merle. 
\end{abstract}

\maketitle

\section{Introduction}
We consider the energy critical focussing wave equation 
\begin{equation}\label{eq:foccrit}
\Box u = u^5,\,\Box = \partial_t^2 - \triangle
\end{equation}
on $\R^{3+1}$, in the radial case. This equation has been intensely studied in a number of recent works: the remarkable series of papers \cite{DKM1} - \cite{DKM4} established a complete classification of all {\it{possible}} type II blow up dynamics, without proving their existence. In the works  \cite{KST}, \cite{DoKr} a constructive approach to actually exhibit and thereby prove the existence of such type II dynamics was undertaken. Recall that a type II blow up solution $u(t, x)$ with blow up time $T_*$ is one for which 
\[
\limsup_{t\rightarrow T_{*}} \big( \|u(t, \cdot)\|_{\dot{H}^1} + \|u_t(t, \cdot)\|_{L_x^2} \big) <\infty 
\]
In \cite{DKM4}, it is demonstrated that such solutions can be described as a sum of dynamically re-scaled ground states
\[
\pm W(x) = \pm \Big(1+\frac{|x|^2}{3}\Big)^{-\frac{1}{2}}
\]
plus a radiation term. In particular, for solutions where only one such bulk term is present, one can write the solution as 
\begin{equation}\label{eq:typeII}
u(t, x) = W_{\lambda(t)}(x) + w(t,x) + o_{\dot{H}^1}(1),\,W_{\lambda} = \lambda^{\frac{1}{2}}W(\lambda x),\; w(t,x)\in \dot{H}^1
\end{equation}
where the ``error'' $(w(t,\cdot), \p_{t}w(t,\cdot))$ converges strongly in $\dot H^{1}\times L^{2}$ as $t\rightarrow T_*$, and we have the dynamic condition 
\begin{equation}\label{eq:rate}
\lim_{t\rightarrow T_*}(T_*-t)\lambda(t) = \infty
\end{equation}
In \cite{KST}, it was shown that such solutions with $\lambda(t) = t^{-1-\nu}$ do exist, where $\nu>\frac{1}{2}$ is arbitrary. This left the question whether for {\it{polynomial}} rates the condition \eqref{eq:rate} is indeed optimal. Here we show that it is.

\begin{thm}\label{thm:Main}
Let $\nu>0$ be given. Then there exists an energy class solution $u(t, x)$, which in fact has regularity $H^{1+\frac{\nu}{2}-}$, of the form \eqref{eq:typeII}, with 
\[
\lambda(t) = t^{-1-\nu}
\]
\end{thm}

Our method of proof is closely modeled on the construction from~\cite{KST}, of which we now recall the main steps:
\begin{enumerate}
\item We write $u_0(t,r)=W_{\lambda(t)}(r)$ and  iteratively modify $u_0$ in the form
\[
u_{2k-1} = u_0 + \sum_{j=1}^{2k-1} v_{j}
\]
so that $u_{2k-1}$ satisfies \eqref{eq:foccrit} up to an error of size $t^N$ as $t\to0+$; here $N$ can be made
as large as desired by taking $k$ large, and the size is measured relative to the energy inside a light cone with tip at $r=0, t=0$. 
\item We seek an exact solution via a perturbation:  $u=u_{2k-1}+\eps$. To solve for $\eps$ we switch to coordinates
$R=\la(t)r$, $\tau=\int_t^\infty \la(s)\, ds=\f{1}{\nu}t^{-\nu}$.  The variable $\tau$ varies in the range $\tau_0\le \tau<\infty$. 
\item  In the new coordinates, the driving linear operator is $$\calL = -\p_{RR} - 5W^4(R) \text{\ \ on\ \ } (0,\infty)$$ We perform a spectral analysis
of the operator, which exhibits a unique and simple negative eigenvalue, as well as continuum spectrum; in addition, there is
a zero energy resonance. The latter renders the spectral measure singular at zero energy. 
\item A contraction argument is set up for $\eps$ with a vanishing condition at $\tau=\infty$. For the contraction it is
important that $N$ in the first step is sufficiently large. 
\end{enumerate}

The first three steps in this paper are essentially the same as in~\cite{KST}. 
It is in the final step that we improve on the procedure in~\cite{KST}. In fact, in Proposition~2.8 of that paper the 
condition $\nu>\f12$ arises in order to make the embedding 
\[
(H^{1+2\alpha}(\R^3))^{5}\subset H^{2\alpha}(\R^3)
\]
for $\nu/2>2\alpha\geq \frac{1}{4}$. 
This was used to control the quintic terms in the construction of the {\it{exact solution}} via iteration and application of a suitable parametrix. 
In fact, the difference shall consist in a more detailed analysis of the {\it{first iterate}} for the exact solution, which we exhibit as a sum of two terms, 
one of which is smooth, the other of which satisfies a good $L^\infty$-bound {\it{near the origin}}. 
This latter feature comes from the fact that the loss of smoothness of the approximate solution occurs precisely on the characteristic light cone, 
and thus one expects the exact solution to be smoother near the spatial origin. 

\section{The approximate solution for a power-law rescaling}

\subsection{Generalities}

In radial coordinates, \eqref{eq:foccrit} becomes
\EQ{\label{u5}
\cL_{quintic} u := u_{tt}-u_{rr}-\f2r u_r - u^5 = 0
}
This equation is known to be locally well-posed in the space $\HH:=\dot H^{1}\times L^{2}(\R^{3})$,
meaning that if $(u(0),u_{t}(0))\in\HH$, then there exists a solution locally in time and continuous in time
taking values in~$\HH$.
Solutions need to be interpreted in the Duhamel sense:
\EQ{
u(t) = \cos(t|\nabla|)f + \frac{\sin(t|\nabla|)}{|\nabla|} g + \int_{0}^{t} \frac{\sin((t-s)|\nabla|)}{|\nabla|} u^{5}(s)\, ds
}
These solutions $\cL_{quintic}(u)=0$  have finite energy:
\[
E(u,u_t) = \int_{\R^{3}} \big[\frac12(u_t^2+|\nabla u|^2)-\f{u^6}{6}\big]\, dx =\const
\]
A special stationary solution is $W(r)=(1+r^2/3)^{-\f12}$.  By scaling, $\la^{\f12}W(\la r)$ is also a solution for any $\la>0$.
We are interested in letting $\la$ depend on time. More precisely, we would like to find solutions $\cL_{quintic} u=0$
of the form
\EQ{\label{cusp}
u(t,r)= W_{\la(t)}(r) + w(t,r),\qquad \la(t)\to\infty \text{\ \ as\ \ }t\to0+
}
and $w$ small in a suitable sense. It suffices to show that $w$ remains small in energy, since  this ensures that the solution blows up at time $t=0$ by the mechanism
of ``energy concentration" at the tip of the light-cone $(t,r)=(0,0)$ (think of solving backwards in time).

\subsection{The bulk term}
For the convenience of the reader,  and in order to correct some minor inaccuracies in~\cite{KST} such as the omission of harmless $\log R$ factors, 
 we redo the iterative construction from that paper which constitutes step~(i) from the four step procedure outlined above. 
 We will henceforth fix~$\la(t)=t^{-1-\nu}$ with $\nu>0$. Set
\EQ{
u_0(t,r)=\la(t)^{\f12} W(r\la(t)) = \la(t)^{\f12} W(R)
}
While $u_{0}$ is very far from being an approximate solution, the authors together with D. Tataru showed in~\cite{KST}  that one can add
successive corrections
\EQ{\label{uk def}
u_{k}=u_{0}+v_{1}+v_{2}+ v_{3} +\ldots +v_{k}
}
so that this function approximately solves~\eqref{u5}.  To be specific, they achieved that $\cL_{quintic} u(t)$ goes to zero like $t^{N}$ in the energy norm
restricted to a light cone where $N$ can be made arbitrarily large by taking~$k$ large. This is an iterative construction.
Moreover, from the point of view of the energy, the functions $v_{j}$ are truly lower order, i.e., they will satisfy
\EQ{\label{vj error}
\int_{r\le t} |\nabla_{t,x} v_{j}(t,x)|^{2}\, dx = O(t^{\nu}) \quad t\to0
}
for all $j\ge1$. In contrast, one of course has 
\[
\int_{r\le t} |\nabla_{t,x} u_{0}(t,x)|^{2}\, dx \simeq 1 \quad t\to0
\]
We shall  now focus on the first two steps of the construction, i.e., $u=u_{0}+v_{1}+v_{2}$.  Let us compute the error resulting from~$u_{0}$.
Define $\cD:=\f12+r\p_r=\f12+R\p_R$.  Then
\EQ{
e_0:=\cL_{quintic} u_0 &= \la^{\f12}(t) \Big [ \Big( \f{\la'}{\la}\Big)^2(t) (\cD^2 W)(R) + \Big(\f{\la'}{\la}\Big)^{\prime}(t) (\cD W)(R)  \Big] \\
t^2 e_0 &=: \la^{\f12}(t) \Big [ \om_1  \f{1-R^2/3}{(1+R^2/3)^{\f32}}   + \om_2 \f{ 9-30R^2+R^4}{(1+R^2/3)^{\f52} }  \Big] \label{e0}
}
Here $\om_j$ are nonzero constants depending on~$\nu$ whose values do not concern us.

\subsection{The first correction}\label{sec:v1}
Then $t^2 e_0 = \la(t)^{\f12}O(R^{2}\lan R\ran^{-3})$   as $R\to\I$.
This error   blows up as $t\to0$ like $t^{-2}$. The goal is now to reduce it --- in fact turn it into an error that vanishes as $t\to0$ ---  by adding corrections
to~$u_{0}$, the first one being~$v_{1}$. We will do this by setting   $\lam^2(t) L_0 v_1 = e_0$ where
\EQ{
L_0 &:= \p_R^2 + \f2R\p_R +5 W^4(R)
}
Note that this is the linearized operator obtained by plugging $u_{0}+v_{1}$ into~\eqref{u5} and discarding~$\p_{t}$ altogether.
While this may seem strange, the idea is to look first at the regime $0<r\ll t$ where $\p_{t}$ should matter less than~$\p_{r}$.
We shall see shortly that $v_{1}$ has the good property that it decays like $(t\la(t))^{-2}$, but it produces errors for the nonlinear PDE that grow in~$r$
too strongly. To remove this growth, we carry out a correction
at the second stage with a differential operator near the light cone $r=t$. There the self-similar variable~$a=\f{r}{t}$ becomes important.

Now we discuss~$v_{1}$ in more detail.
A fundamental system of $L_0$ is
\EQ{\label{L0 fund}
\fy_1(R):=\f{1-R^2/3}{(1+R^2/3)^{\f32}},\quad \fy_2(R):=\f{1-2R^2 +R^4/9}{R(1+R^2/3)^{\f32}}
}
The operator
\EQ{\label{tilde L0}
\tilde L_0 = R L_0\, R^{-1} = \p_R^2  +5 W^4(R)
}
has a fundamental system
\EQ{\label{tilde fundsys}
\tilde \fy_1(R) &:=R\fy_1(R)=\f{R(1-R^2/3)}{(1+R^2/3)^{\f32}} =  \tilde\psi_1 (R^{-2})\\
\tilde \fy_2(R) &:= R\fy_2(R) = \f{1-2R^2 +R^4/9}{(1+R^2/3)^{\f32}} = R \tilde\psi_2 (R^{-2})
}
The right-hand sides here are for large~$R$, and the $\tilde\psi_j$ are analytic around~$0$.
The Wronskian is
\EQ{\label{W}
\tilde\fy_1'(R)\tilde \fy_2(R)-\tilde \fy_1(R)\tilde \fy_2'(R) = 1
}
Define $\mu(t):=t\la(t)$, and
\EQ{
\mu^2(t) L_0 v_1 = t^2 e_0,\quad v_1(0)=v_1'(0)=0
}
We claim that
\EQ{
v_1(t,r) =  \mu^{-2}(t) L_0^{-1} t^2  e_0  =  \la^{\f12}(t)  \mu^{-2}(t) O(R) \text{\ \ as\ \ }R\to\I
}
To be more specific,
write
\EQ{
t^2\, e_0 = \la^{\f12}(t)\big( \om_1 \, g_1(R) + \om_2 \, g_2(R) \big)
}
see \eqref{e0}.  Note that the $g_j$ are of the form
\EQ{\label{gj analytic}
g_j(R)= R^{-1} \phi_j(R^{-2}) \quad R\gg 1
}
where $\phi_j$ is analytic around $0$.
Then  $L_0 f_j = g_j$ with $f_j(0)=f_j'(0)=0$ satisfies
\EQ{\label{int}
f_j(R) = R^{-1} \Big( \tilde\fy_1(R) \int_0^R \tilde\fy_2(R')R' g_j(R')\, dR' -  \tilde\fy_2(R) \int_0^R \tilde\fy_1(R')R' g_j(R')\, dR'\Big)
}
for $j=1,2$. The one checks that
\EQ{\label{fasymp}
f_j(R) &= b_{1j} R + b_{2j} + b_{3j}\f{\log R}{R} + O(1/R)\text{\ \ as\ \ }R\to\I \\
f_j(R) &= c_{1j} R^2 +O(R^4) \text{\ \ as\ \ }R\to 0
}
In fact,   around $R=0$ the $f_j(R)$ are even analytic functions, whereas  around $R=\I$ one has the representation
\EQ{\label{fasymp infty}
f_j(R) &= R(b_{1j} + b_{2j} R^{-1}+ R^{-2}\log R\;\fy_{1j}(R^{-2}) +  R^{-2} \fy_{2j}(R^{-1})  ) \\
&=: R(F_j(\rho) + \rho^{2} G_{j}(\rho^2)\log \rho)
}
where $\fy_{1j}, \fy_{2j}$ and $F_j, G_j$ are analytic around zero, with $\rho:=R^{-1}$.  This follows from~\eqref{tilde fundsys}, \eqref{gj analytic}, and~\eqref{int}.
For future reference, we remark that the structure in~\eqref{fasymp infty} is preserved under application of~$\cD$.
In particular,
\EQ{
v_1(t,r) = \la^{\f12}(t)  \mu^{-2}(t)  (\om_1 f_1(R) + \om_2 f_2(R)) =: \la^{\f12}(t) \mu^{-2}(t)  f(R)
}
Define
$$u_1:= u_0+ v_1 = \la^{\f12}(t)\big( W(R) + \mu^{-2}(t) f(R) )$$
In view of~\eqref{fasymp}, and $R\le \mu$ (recall that we are inside of the light cone $r\le t$)
\EQ{\label{u1 R}
u_1(t,r) = \la^{\f12}(t) O(R^{-1}) \quad R\ge1
}
uniformly in $0<t<1$; moreover, we may apply $t\p_t$  
any number of times without affecting this asymptotic property.
Finally, $\la(t)^{-\f12} u_1(t,r)$ is an even analytic function around $R=0$.

\subsection{The  error from $u_1$}
Set $e_1:= \cL_{quintic} (u_1)$.
Then
\EQ{\label{e1def}
e_1 &=  \p_t^2 v_1 - 10 u_0^3 v_1^2 -10 u_0^2 v_1^3 -5 u_0 v_1^4 - v_1^5
}
One has
\EQ{\label{e1}
t^2 \la^{-\f12}(t) e_1 &= \la^{-\f12}(t) ((t\p_t)^2 - t\p_t) \big(  \la^{\f12}(t) w_1(t,r\lam(t)) \big) - \mu^2(t)  (10 W^3(R) w_1^2(t,R) \\
&\qquad + 10 W^2(R) w_1^3(t,R) + 5 W(R) w_1^4(t,R) + w_1^5(t,R))
}
where $w_1(t,R)= \mu^{-2}(t)  f(R)$. Then the nonlinearity in~\eqref{e1} is
\EQ{
& \mu^2(t)(10 W^3(R) w_1^2(t,R)  + 10 W^2(R) w_1^3(t,R) + 5 W(R) w_1^4(t,R) + w_1^5(t,R))  \\
& = \mu^{-2}(t) \big(10 W^3(R)   f^2(R)   + 10 W^2(R) \mu^{-2}(t)   f^3(R) \\
&\quad + 5 W(R)  \mu^{-4}(t)   f^4(R) + \mu^{-6}(t)   f^5(R) \big)
}
whereas
\EQ{
&\la^{-\f12}(t) ((t\p_t)^2 - t\p_t) \big(  \la^{\f12}(t)w_1(t,r\lam(t)) \\&=\left(\left(t\p_t + \frac{t\la'(t)}{\la(t)}\cD\right)^2 - \left(t\p_t + \frac{t\la'(t)}{\la(t)}\cD\right) \right)w_1(t,R) \\
&= \left(\left(t\p_t - (1+\nu) \cD\right)^2 - \left(t\p_t - (1+\nu)\cD\right) \right)w_1(t,R)
}
Now $\mu(t)= t^{-\nu}$ whence
\EQ{
&\mu^2(t) \left(t\p_t + \frac{t\la'(t)}{\la(t)}\cD\right) \mu^{-2}(t)   f(R) = ( 2\nu  -(1+\nu) \cD )   f(R)
}
Note that this is again of the form $ f(R)$ with $f$ as in~\eqref{fasymp}, \eqref{fasymp infty}.
Thus we can write
\EQ{\label{e1*}
t^2 \la^{-\f12}(t) e_1(t,r)
&= \mu^{-2}(t) \Big(   f(R) - \big(10 W^3(R)   f^2(R)  + 10\mu^{-2}(t)  W^2(R)  f^3(R)  \\ &
 \qquad \qquad + 5 \mu^{-4}(t) W(R)    f^4(R) + \mu^{-6}(t)   f^5(R) \big) \Big)
}
We let $a=\frac{r}{t}=\frac{R}{\mu}=Rb$, $b:=\mu^{-1}$ and isolate those terms   in~\eqref{e1*} which do not decay for large~$R$.
Since we are working inside of the light-cone, we have $0\le a\le 1$.
Now, abusing notation somewhat,
\begin{align}
\mu^{-2}(t)  f(R) &= b^2 R (F(\rho)+ \rho^{2} G(\rho^2)\log \rho)  = ba (F(\rho)+ \rho^{2} G(\rho^2)\log \rho)  \nn  \\
\mu^{-2}(t)  W^3(R)  f^2(R)  &=  b^2 R^{-3} \Omega(\rho^2) R^2 (F(\rho)+ \rho^{2} G(\rho^2)\log \rho)^2 \nn  \\
&=  b^2 R^{-1}    (F(\rho)+ \rho^{2} F(\rho)\log \rho + \rho^{4} G(\rho^2)\log^2 \rho ) \label{mess1} \\
\mu^{-4}(t)  W^2(R)  f^3(R)  &=  b^4 R^{-2}\Omega(\rho^2) R^3 (F(\rho)+ \rho^{2} G(\rho^2)\log \rho)^3 \nn \\
&=  b^3 a (F(\rho)+ \rho^{2} F(\rho)\log \rho + \rho^{4} F(\rho)\log^2 \rho + \rho^{6} G(\rho^2)\log^3 \rho)   \nn
\end{align}
where $F,G$ can change from line to line.
Similarly,
\EQ{ \label{mess2}
\mu^{-6}(t)  W(R)  f^4(R)  &= b^6 R^{-1} \Omega(\rho^2)  R^4 (F(\rho)+ \rho^{2} G(\rho^2)\log \rho)^4\\
&= b^3 a^3   (F(\rho)+ \rho^{2} F(\rho)\log \rho  + \rho^{4} F(\rho)\log^2 \rho \\
&\qquad + \rho^{6} F(\rho)\log^3 \rho + \rho^{8} G(\rho^2)\log^4 \rho) \\
\mu^{-8}(t)    f^5(R)  &= b^8 R^5  (F(\rho)+ \rho^{2} G(\rho^2)\log \rho)^5 \\
&= b^3 a^5 (F(\rho)+ \rho^{2} F(\rho)\log \rho  + \rho^{4} F(\rho)\log^2 \rho \\
&\qquad + \rho^{6} F(\rho)\log^3 \rho + \rho^{8} F(\rho)\log^4 \rho + \rho^{10} G(\rho^2)\log^5 \rho)
}
From \eqref{mess1}, \eqref{mess2} we extract the
leading order
\EQ{ \label{e10*}
t^2 \la^{-\f12}(t) e_1^0(t,r)  &:= \mu^{-1}(t) ( c_1 a + c_2 b+ ( c_3 a   + c_4 a^3   + c_5  a^5 )  b^2 )
}
Indeed, from the first line in~\eqref{mess1} we extract $ba(F(0)+\rho F'(0))= ba c_1 +  b^2 c_2$,
whereas from the fifth we extract $b^3 a F(0)$. From the second line in~\eqref{mess2} we retain $b^3 a^3 F(0)$,
and from the fifth one~$b^3 a^5 F(0)$.
The point here is that with this choice of $e_1^0$ one obtains  a decaying error as $R\to\I$
\EQ{ \label{ediff}
&t^2 \la^{-\f12}(t) ( e_1- e_1^0)(t,r) \\
& = \mu^{-2}(t) \Big[ \frac{\log R}{R} \Phi_1(a,b,\rho\log\rho,\rho) + \frac{1}{R} \Phi_2(a,b,\rho\log\rho,\rho)\Big]
}
where $\Phi_j(a,b,u,v)$ are polynomials in $a,b$ and analytic in $u,v$ near $(0,0)$. Writing $b=\frac{a}{R}$ we may
delete the terms involving~$b^2=ba/R$ on the right-hand side of~\eqref{e10*}, since they are of the form~\eqref{ediff}.
Thus, it suffices to consider the simpler leading error
\EQ{ \label{e10}
t^2 \la^{-\f12}(t) e_1^0(t,r)  &:=c_1 a  \mu^{-1}(t) + c_2\mu^{-2}(t) = c_1 ab + c_2 b^2 
}

\subsection{The second correction}
Now we would like to solve the corrector problem ``near $r=t$", i.e.,
\EQ{
t^2 \big(v_{tt}-v_{rr}-\f2r v_r   \big) = -t^2 e_1^0  
}
Note that we have discarded the nonlinearity on the left-hand side since it decays near~$r=t$.
This is designed exactly so as to remove the growth in~$R$.
We seek a solution in the form
\EQ{\label{v2}
v(t,r) = \la(t)^{\f12}\big( \mu^{-1}(t) q_1(a) + \mu^{-2}(t) q_2(a) \big)
}
with boundary conditions $q_1(0)=0, q_1'(0)=0$ and $q_2(0)=0$, $q_2'(0)=0$.
These translate into the boundary conditions $v(t,0)=0$, $\p_r v(t,0)=0$.  This will essentially be the
function~$v_2$.
In view of
\[
\la(t)^{-\f12} \mu^\al \p_t\:  \la(t)^{\f12} \mu^{-\al} =  \p_t + \f{(\al-\f12)\nu-\f12}{t}
\]
we are reduced to the system
\EQ{\label{w1}
&t^2 \Big(-\big(\p_t + \f{\nu-1}{2t}\big)^2 +\p_{rr}+\f2r \p_r  \Big) q_1(a)  =   c_1 a }
and
\EQ{\label{w0}
&t^2 \Big(-\big(\p_t +  \f{3\nu-1}{2t}\big)^2 +\p_{rr}+\f2r \p_r   \Big) q_2(a)  =   c_2
}
Now
\EQ{\nn
& t^2 \Big(-\big(\p_t + \f{\beta}{t}\big)^2 +\p_{rr}+\f2r \p_r  \Big) f(a) \\
& = \big( (1-a^2)\p_a^2 + (2(\beta-1)a+2a^{-1})\p_a - \beta^2 + \beta  \big) f(a)
}
Define
\EQ{\label{Lbg}
L_{\beta}:= (1-a^2)\p_a^2 + (2(\beta-1)a+2a^{-1})\p_a - \beta^2 + \beta
}
The system \eqref{w1}, \eqref{w0} therefore becomes
\EQ{\label{q1q2}
L_{\f{\nu-1}{2}} \; q_1 = c_1 a,\quad L_{\f{3\nu-1}{2}}\; q_2 = c_2
}
with boundary conditions $q_1(0)=0, q_1'(0)=0$ and $q_2(0)=0$, $q_2'(0)=0$.

Since for any integer $n\ge2$
\[
L_\beta\, a^n = n(n+1) a^{n-2}  + (2n(\beta-1) - n(n-1) +\beta -\beta^2) a^n
\]
we see that \eqref{q1q2} has power series solutions
\EQ{\label{power}
q_1(a)=\sum_{j=1}^\I c_{1j} \, a^{2j+1},\qquad q_2(a)=\sum_{j=1}^\I c_{2j} \, a^{2j}
}
convergent in $|a|<1$ and unique. In fact, $c_{21}=\f{c_2}{6}$ and $c_{11}=\f{c_1}{12}$.

Next, we determine the behavior of these functions as $a\to1-$. We remark that in our case, always $\beta>-\f12$.  There exists a fundamental system of $L_\beta$ of the form
\EQ{
\psi_1(a) &=1 +\sum_{\ell=1}^\I d_\ell\, (1-a)^\ell, \\
 \psi_2(a) &= (1-a)^{\beta+1}\big[1 +\sum_{\ell=1}^\I \tilde d_\ell\, (1-a)^\ell \big] = (1-a)^{\beta+1} \psi_2^0(a)
}
provided $\beta\not\in \Z_0^+$, see~\cite[Lemma 3.6]{KST}. In that case these series define entire functions.  On the other hand, if $\beta\in \Z_0^+$,
then $\psi_1$ is modified to
\[
\psi_1(a) =1 +\sum_{\ell=1}^\I d_\ell \, (1-a)^\ell+ c\,\psi_2(a)\log(1-a)
\]
with a unique choice of~$c$.    Notice that in this case the singularity of~$\psi_1$ near $a=1$
is no worse than~$(1-a)^{\beta+1}\log(1-a)$.

In either case the Wronskian
\[
W = \psi_1' \, \psi_2 - \psi_1\,\psi_2'
\]
satisfies
\[
W'(a)=- \f{2(\beta-1)a+2a^{-1}}{1-a^2} W(a)
\]
whence
\[
W(a)=k (1-a^2)^\beta a^{-2}, \quad k\ne0
\]
We define the Green function
\[
G_\beta(a,a'):= \psi_1(a)\psi_2(a')-\psi_1(a')\psi_2(a)
\]
Then a particular solution of $L_\beta\, q=f$ is given by
\EQ{\label{qint}
q(a) &= \int_{0}^a G_\beta(a,\ti a) W(\ti a)^{-1} (1-\ti a^2)^{-1}   f(\ti a)\, d \ti a \\
&= k^{-1} \int_{0}^a G_\beta(a,\ti a)  (1-\ti a^2)^{-(\beta+1)}\, \ti a^2  f(\ti a)\, d \ti a \\
& = k^{-1} \psi_1(a)  \int_{0}^a \psi_2^0(\ti a)\, (1+\ti a)^{-\beta-1}\ti a^2  f(\ti a)\, d \ti a \\&\qquad - k^{-1} \psi_2(a)  \int_{0}^a \psi_1(\ti a)(1-\ti a^2)^{-(\beta+1)}\, \ti a^2  f(\ti a)\, d \ti a
}
Returning to~\eqref{q1q2} we need to set $f(a)=1$ and $f(a)=a$ here.    Note that $q(a)$ as given by~\eqref{qint} satisfies
the boundary conditions $q(0)=q'(0)=0$, whence it agrees with $q_2$ and~$q_1$, respectively, as given by~\eqref{power}.

Let us first assume that $\beta\not\in \Z_0^+$. Then the term on the third line of~\eqref{qint} is analytic near~$a=1$.
To analyze the behavior of the expression on the fourth line as $a\to 1$, we note that up to an analytic factor near $a=1$ it
equals
\EQ{\label{2ndint}
(1-a)^{\beta+1} \int_0^a (1-\ti a)^{-\beta-1}\, h(\ti a) \, d\ti a
}
where $h$ is analytic in a neighborhood of $[0,1]$.  By inspection, \eqref{2ndint} is of the form
\[
(1-a)^{\beta+1}\big( (1-a)^{-\beta} H_1(a) + c) H_2(a)
\]
where $c$ is a constant and $H_1, H_2$ are analytic near $[0,1]$. It  follows that $\beta\not\in \Z_0^+$, which
means that $\nu$ is neither an odd positive integer nor of the form $\f{2n+1}{3}$ with $n\in \Z_0^+$, one has that
 $q_1(a), q_2(a)$ are analytic in the disk $|a|<1$ and near $a=1$ they are of the form
 \EQ{\label{Q}
 Q_1(a) + (1-a)^{\beta+1} Q_2(a)
 }
where $Q_j$ are analytic near $a=1$.  These functions have the property that after applying~$\p_a$ they remain in~$L^2(0,1)$ due to $\beta>-\f12$.
 Evidently, they also become smoother as $\beta$ increases, but they are never infinitely smooth (since $\beta$ is not an integer).

On the other hand, if $\beta\in \Z_0^+$, then the representation~\eqref{Q} needs to be modified with logarithmic factors $\log(1-a)$.
We leave these details to the reader.

Using  $a=R\mu^{-1}$ we may rewrite~\eqref{v2} in the form
\EQ{\label{v2 1}
v_2^0(r,t):= \frac{\la(t)^{\f12}}{\mu^2(t)} \big( R \tilde q_1(a)+ q_2(a))
}
where we have set $\tilde q_1(a):= a^{-1} q_1(a)$. This ensures that both $\ti q_1$ and $q_2$ have even
power series in~$a$ around $a=0$. Also note that these functions are~$O(a^2)$ as $a\to0$.   We make one more adjustment:  in~\eqref{v2 1} one has an odd expansion
around $R=0$, namely just the linear term~$R$.  We prefer to modify~\eqref{v2 1} as follows so as to retain the
even expansion at~$R=0$:
\EQ{\label{v2 2}
v_2(r,t):= \frac{\la(t)^{\f12}}{\mu^2(t)} \big( R^2 \lan R\ran^{-1} \tilde q_1(a)+ q_2(a))
}
Note that for large~$R$ this captures the $R$ growth of~\eqref{v2 1}, and the next order correction is $R^{-1}$.
 With this definition of~$v_2$ we set  $u_2:= u_1+v_2=u_0+v_1+v_2$. By construction, $v_2(t,r)$ is analytic in~$R$
 around $R=0$ with an even expansion.  Finally, \eqref{u1 R} remains valid for~$u_2$ as well. In other words, $u_0$
 gives the main shape of the profile as a  function of~$R$.

 \subsection{The error from $u_2$}
 We define
 \EQ{
 e_2 &:= \cL_{quintic} (u_2) = \cL_{quintic} (u_1+ v_2) \\
 &=  \cL_{quintic} (u_1)  + u_1^5 - (u_1+v_2)^5 +(  \p_{tt} -\p_{rr} - \f{2}{r}\p_r )v_2\\
 &= e_1 - e_1^0 - 5u_1^4 v_2 - 10 u_1^3 v_2^2 - 10 u_1^2 v_2^3 - 5 u_1 v_2^4 - v_2^5 \\
 & \qquad +  (  \p_{tt} -\p_{rr} - \f{2}{r}\p_r )(v_2-v_2^0)
 }
 We determine $t^2 \la(t)^{-\f12} e_2$. First, from~\eqref{ediff}
 \EQ{ \label{ediff 2}
&t^2 \la^{-\f12}(t) ( e_1- e_1^0)(t,r) \\
& = \mu^{-2}(t) \Big[ \frac{\log R}{R} \Phi_1(a,b,\rho\log\rho,\rho) + \frac{1}{R} \Phi_2(a,b,\rho\log\rho,\rho)\Big]
}
for $R\ge1$. For $|R|<1$ we read off from~\eqref{e1*} and~\eqref{e10} that
 \EQ{ \label{ediff 3}
&t^2 \la^{-\f12}(t) ( e_1- e_1^0)(t,r)  = O(\mu^{-2}(t))
}
This holds uniformly for small times, and $t\p_t$ can be applied any number of times without changing this asymptotic behavior as $R\to0$.

Next, for large~$R$
\[
t^2 \la^{-\f12}(t) u_1^4 v_2 = O(R^{-3} a^2) = O(R^{-1} \mu(t)^{-2})
\]
which is of the form~\eqref{ediff 2}.
The final nonlinear term contributes
\[
t^2 \la^{-\f12}(t)  v_2^5 = \mu^{-8} O(R^5) = \mu^{-2} R^{-1} O(\mu^{-6} R^6) = O(R^{-1} \mu(t)^{-2})
\]
We leave it to the reader to verify that the other nonlinear terms behave in the same fashion.  For small~$R$,
the nonlinear terms are $O(\mu^{-2}(t) R^2)$.

 The difference $v_2-v_2^0$ contributes this: for $R>1$,
 \EQ{
& \la(t)^{-\f12} t^2 (  \p_{tt} -\p_{rr} - \f{2}{r}\p_r ) \f{\la(t)^{\f12}}{\mu(t)^2} \big( (R-R^2\lan R\ran^{-1}) \ti q_1(a)\big) = O(R^{-1} \mu(t)^{-2})
 }
 Here we used that $\ti q_1(a) = O(a^2)$ for small $a$. For small $R$ this term is~$O(\mu^{-2}(t) R)$.
By inspection, $e_2 = \cL_{quintic} (u_2)$ has an even analytic expansion around $R=0$, and by the preceding we gain a factor $\mu^{-2}$
for all~$R$, and the decay is at least $\frac{\log R}{R} $ as~$R\to\I$.

  \subsection{Iterating the construction: the next  corrections $v_3$ and $v_4$}

 We now return to Section~\ref{sec:v1} in which we constructed~$v_1$ from~$e_0$. Here we need to determine~$v_3$ from~$e_2$
 via the same route, i.e., by solving $\la(t)^2 L_0 v_3 = e_2$ with zero Cauchy data at~$R=0$.
The only essential difference between $e_0$ and $e_2$ is that the latter looses a $\log R$ in terms of the $R^{-1}$ decay.
The dependence on~$a,b$ makes no difference in terms of the asymptotic behavior or the construction, however. Therefore, the same
construction as before yields
\[
v_3(t,r) = \frac{\la(t)^{\f12}}{\mu(t)^4} O(R\log R)
\]
 as $R\to\I$. Moreover, the asymptotic expansion involves terms with both $\log R$ and $\log^2 R$.
 Around $R=0$ we have an even Taylor expansion starting off with~$R^2$.   The
 \[
 u_3:= u_0 + v_1 + v_2 + v_3
 \]
  still obeys the decay law \eqref{u1 R}; in fact, any finite number of powers of $\log R$ can be lost since they are more than
  made up for by the gain of an extra power of~$\mu^{-2}$ in~$v_3$. In this sense, $v_3$ is of strictly smaller order than both $v_1$ and~$v_2$
 which can be comparable to~$u_0$ on the light-cone $r=t$.
  For $v_4$, one repeats the same construction that lead to~$v_2$ above designed to obtain a decaying error.   The error will be on
  the order of~$\mu^{-4}$.

 The process can then be repeated any number of times. In~\cite{KST} the authors, together with D. Tataru, formalized this using various function algebras
 designed to hold the $v_j$ the associated errors $e_j$. The details are as follows. 
 We have 
  \[
  u_k = v_{k} + u_{k-1}
  \]
  The error at step $k$ is
  \[
  e_k = (-\partial_t^2 + \partial_r^2 +\frac{2}r \partial_r) u_k
  + u_k^5
  \]
  we
  define $v_k$ by
  \begin{equation}
    \left(\partial_r^2 +\frac{2}r \partial_r + 5u_0^4\right) v_{2k+1} +
    e_{2k}^0=0
    \label{vkodd}\end{equation} respectively
  \begin{equation}
    \left(-\partial_t^2+\partial_r^2 +\frac{2}r \partial_r \right) v_{2k} +
    e_{2k-1}^0=0
    \label{vkeven}\end{equation} both equations having zero Cauchy
  data  at $r=0$.  Here at each
  stage the error term $e_k$ is split into a principal part and a
  higher order term (to be made precise below),
  \[
  e_k = e_k^0 + e_k^1
  \]
  The successive errors are then computed as
  \[
  e_{2k} = e_{2k-1}^1 + N_{2k} (v_{2k}), \qquad e_{2k+1} = e_{2k}^1 -
  \partial_t^2 v_{2k+1} + N_{2k+1} (v_{2k+1})
  \]
  where
  \begin{equation}
    N_{2k+1}(v) =5(u_{2k}^4-u_0^4) \,v +
  10 u_{2k}^3 \,v^2 + 10 u_{2k}^2 \,v^3 + 5u_{2k} \,v^4 + v^5
    \label{eodd}\end{equation}
    respectively
  \begin{equation}
    N_{2k}(v) = 5 u_{2k-1}^4v +
  10 u_{2k-1}^3 v^2 + 10 u_{2k-1}^2 v^3 + 5u_{2k-1} v^4 + v^5
    \label{eeven}\end{equation}
The function spaces are as follows. First, the one relevant to the $L_{0}$ iteration.  

 \begin{defi}\label{def:Smkell}
  $S^m(R^k\,(\log R)^{\ell} )$ is the class of smooth functions
  $v:[0,\infty) \to \R$ with the following properties:

  (i) $v$ vanishes of order $m$  and $R^{-m} v$ has an even
  Taylor expansion at $R=0$. 

  (ii) $v$ has a convergent expansion near $R=\infty$ of the form
  \[
  v(R) = \sum_{i=0}^\infty \sum_{j=0}^{\ell+i} c_{ij}\, R^{k-i} (\log R)^{j} 
  \]
\end{defi}

The importance of even expansions in $R$ near zero lies with the fact that
only those correspond to smooth functions in~$\R^3$. For the same
reason, we will work with even~$m$. 
Second, we introduce the space arising from the Sturm-Liouville problem near $a=1$. 

\begin{defi}\label{def:Q}     
  We define $\mathcal Q$ to be the algebra of continuous functions $q:[0,1] \to
  \R$ with the following properties:

  (i) $q$ is analytic in $[0,1)$ with an even expansion at $0$ and with $q(0)=0$.

  (ii) Near $a=1$ we have an  expansion of the
  form
  \[
\begin{split}
  q(a) = q_0(a) + \sum_{i =1}^\infty  (1-a)^{ \beta(i) +1}\sum_{ j=0}^{\infty}
    q_{ij}(a) (\ln (1-a))^j
  \end{split}
  \]
  with analytic coefficients $q_0$, $q_{ij}$; if $\nu$
  is irrational, then $q_{ij}=0$ if~$j>0$.  The $\beta(i)$ are of the form
  \EQ{\label{menge}
  \sum_{k\in K} \big( (2k-3/2)\nu-1/2 \big) + \sum_{k\in K'} \big( (2k-1/2)\nu-1/2 \big)  
  }
  where $K,K'$ are finite sets of positive integers. 
  Moreover, only finitely many of the $q_{ij}$ are nonzero. 
\end{defi}

We remark that the exponents of $1-a$ in the above series all exceed~$\f12$ 
because of $\nu>0$.  For the errors $e_k$ we
introduce

\begin{defi}
   $\mathcal Q'$ is the space of continuous functions $q:[0,1) \to
  \R$ with the following properties:

  (i) $q$ is analytic in $[0,1)$ with an even expansion at $0$. 

  (ii) Near $a=1$ we have an expansion of the form
  \[\begin{split}
  q(a) =  q_0(a) + \sum_{i =1}^\infty  (1-a)^{ \beta(i)}\sum_{ j=0}^{\I}
    q_{ij}(a) (\ln (1-a))^j
  \end{split}
  \]
  with analytic coefficients $q_0$, $q_{ij}$, of which only finitely many are nonzero. The $\beta(i)$ are as above. 
\label{def:Qp}\end{defi}

By construction, $\mathcal Q\subset \mathcal Q'$. The family
$\cQ'$ is obtained by applying $a^{-1}\partial_a$ to
the algebra~$\cQ$. The exact number of $\log(1-a)$ factors can of course
be determined, but is irrelevant for our purposes.

\begin{defi}
  a) $S^m(R^k (\log R)^{\ell} ,\mathcal Q)$ is the class of analytic
  functions $v:[0,\infty) \times [0,1]\times [0,b_0] \to \R$ so that

  (i) $v$ is analytic as a function of $R,b$, and 
  \[
  v: [0,\infty) \times [0,b_0] \to \mathcal Q
  \]

  (ii) $v$ vanishes of order $m$ relative to~$R$,  and $R^{-m} v$ has an even
  Taylor expansion at $R=0$.

  (iii) $v$ has a convergent expansion at $R=\infty$,
  \[
  v(R,a,b) = \sum_{i=0}^\infty \sum_{j=0}^{\ell+i} c_{ij}(a,b) \, R^{k-i} (\log R)^{j}
  \]
  where the coefficients $c_{ij}(\cdot,b) \in \cQ$, and $c_{ij}(a,b)$  are analytic in~$b\in [0,b_{0}]$ for
  all~$0\le a\le1$. 

  b) $\IS^m(R^k (\log R)^{\ell},\mathcal Q)$ is the class of analytic
  functions $w$ on the cone $\mathcal C_0$ which can be represented as
  \[
  w(r,t) = v(R,a,b), \qquad v \in S^m(R^k (\log R)^{\ell} ,\mathcal Q)
  \]
  
  The same holds with $\cQ'$ in place of $\cQ$. 
\end{defi}

As in the first two steps we shall exploit that  $R,a,b$ are dependent variables, 
by switching from one representation
to another as needed. We shall prove by induction that the
successive corrections $v_k$ and the corresponding error terms $e_k$
can be chosen with the following properties. There exist  increasing sequences $m_{k}, p_{k}, q_{k}$
of nonegative integers  with $m_{1}=p_{1}=0, q_{1}=1$ so that for each $k\ge1$,   

\begin{align}
  v_{2k-1} &\in \frac{\lambda^{\frac12}}{\mu(t)^{2k}} \IS^2(R \, (\log R)^{m_{k}},\cQ)
  \label{v2k-1}\\
  t^2 e_{2k-1} &\in \frac{\lambda^{\frac12}}{\mu(t)^{2k}} \IS^0(R\, (\log R)^{p_{k}}, \cQ')
  \label{e2k-1}\\
  v_{2k} &\in \frac{\lambda^{\frac12}}{\mu(t)^{2k+2}} \IS^2(R^3\, (\log R)^{p_{k}},\cQ) \label{v2k}\\
  t^2 e_{2k} &\in \frac{\lambda^{\frac12}}{\mu(t)^{2k}} \big[\IS^0(R^{-1} \, (\log R)^{q_{k}} ,\cQ)
   + b^2\IS^0(R \, (\log R)^{q_{k}} ,\cQ')    \big]
  \label{e2k}\end{align}

One can of course easily determine the optimal choice of $m_{k}, p_{k},q_{k}$ from the algorithm outlined
below, but this is of no significance. 
We now inductively verify these claims. 

\medskip

{\bf Step 0:} {\em The analysis at $k=0$}

\smallskip
\noindent We observed in~\eqref{e0} that 
\begin{equation}\label{e20}
t^2 e_0 \in
\lambda^{\frac12} IS^0(R^{-1})
\end{equation}
as claimed. Now assume we know the above relations hold up to $k-1$
with $k\ge1$, and we show how to construct $v_{2k-1}$, respectively
$v_{2k}$, so that they hold for the index~$k$.

\medskip

{\bf Step 1:} {\em Begin with $e_{2k-2}$ satisfying \eqref{e2k}
  or~\eqref{e20} and choose $v_{2k-1}$ so that \eqref{v2k-1} holds.}

\smallskip \noindent If $k=1$, then define $e_0^0:= e_0$. If $k>1$,
we use \eqref{e2k} to write
\[
  e_{2k-2} = e_{2k-2}^0 + e_{2k-2}^1
\]
where
\EQ{
& t^2 e_{2k-2}^0 \in \frac{\lambda^{\frac12}}{\mu(t)^{2k-2}}
\IS^0(R^{-1} (\log R)^{q_{k-1}} ,\cQ), \\
& t^2 e_{2k-2}^1 \in
\frac{\lambda^{\frac12}}{\mu(t)^{2k}} \IS^0(R (\log R)^{q_{k-1}} ,\cQ')
}
We note that the term $e_{2k-2}^1$ can be included in $e_{2k-1}$,
cf.~\eqref{e2k-1}. We define $v_{2k-1}$ as in \eqref{vkodd}
neglecting the $a$ dependence of $e_{2k-2}^0$. In other words, we choose to treat~$a$
as a parameter and the error resulting from this choice will then be incorporated
in~$e_{2k-1}$. 

  Changing variables to $R$ in
\eqref{vkodd} we need to solve the equation
\[
\mu(t)^2 L_{0} v_{2k-1} = - t^2 e_{2k-2}^0 \in
\frac{\lambda^{\frac12}}{\mu(t)^{2k-2}} \IS^0(R^{-1} (\log R)^{q_{k-1}}, \cQ)
\]
where the operator $L_{0}$ is the one from above. 
Then \eqref{v2k-1} is a consequence of the following ODE lemma.
We remark that the statement of Lemma~\ref{lem:Lsolve} is not optimal
with respect to the number of logarithms in~$R$. But for the sake of
simplicity we choose this formulation. 

\begin{lem} \label{lem:Lsolve}
The solution $v$ to the equation
  \[
  L_{0} v= f \in S^0(R^{-1}(\log R)^\ell ), \qquad v(0) = v'(0)=0
  \] with integer $\ell\ge0$ 
  has the regularity
  \EQ{\label{v hold}
  v \in S^2(R\,(\log R)^\ell )  
  }
\end{lem}

\begin{proof}
  This follows form the representation~\eqref{tilde fundsys}, \eqref{gj analytic}, and~\eqref{int}. Indeed, 
  \EQ{\label{scheiss1}
v(R) &= R^{-1} \Big( \tilde\fy_1(R) \int_0^R \tilde\fy_2(R')R' f(R')\, dR' -  \tilde\fy_2(R) \int_0^R \tilde\fy_1(R')R' f(R')\, dR'\Big) \\
&= R^{-1} \tilde\psi_1 (R^{-2}) \int_{R_{0}}^{R} y \tilde\psi_2 (y^{-2})  \sum_{i=0}^\infty \sum_{j=0}^{\ell+i} c_{ij}\, y^{-i} (\log y)^{j} \, dy \\
&\qquad - \tilde\psi_2(R^{-2}) \int_{R_{0}}^R \tilde\psi_1(y^{-2})  \sum_{i=0}^\infty \sum_{j=0}^{\ell+i} c_{ij}\, y^{-i} (\log y)^{j} \, dy  \\
&\qquad + c_{1}\phi_{1}(R)+c_{2}\phi_{2}(R) 
}
 The first integral in \eqref{scheiss1}  contributes
 \EQ{\nn
 & R^{-1} \tilde\psi_1 (R^{-2})\int_{R_{0}}^{R}  \sum_{n=0}^{\I} a_{n}\, y^{-2n}  \sum_{i=0}^\infty \sum_{j=0}^{\ell+i} c_{ij}\, y^{1-i} (\log y)^{j} \, dy \\
 & = \const\,\phi_{1}(R)+ \tilde\psi_1 (R^{-2})\sum_{n=0}^{\I} \sum_{i=0}^\infty \sum_{j=0}^{\ell+i} \chi_{[1-2n-i\ne-1]}a_{nij}\,  R^{1-i-2n} (\log R)^{j} \\
 &\quad + R^{-1} \tilde\psi_1 (R^{-2}) \sum_{n=0}^{\I} \sum_{i=0}^\infty \sum_{j=0}^{\ell+i} \chi_{[1-2n-i=-1]}a_{nij}\,   (\log R)^{j+1} 
 }
  which lies in $S^{2}(R \,(\log R)^\ell )$, whereas the second one contributes
   \EQ{\nn 
 &  \tilde\psi_2 (R^{-2})\int_{R_{0}}^{R}  \sum_{n=0}^{\I} b_{n}\, y^{-2n}  \sum_{i=0}^\infty \sum_{j=0}^{\ell+i} c_{ij}\, y^{-i} (\log y)^{j} \, dy \\
 & = \const\,\phi_{2}(R)+ \tilde\psi_2 (R^{-2}) \sum_{n=0}^{\I} \sum_{i=0}^\infty \sum_{j=0}^{\ell+i} \chi_{[-2n-i\ne-1]}b_{nij}\,  R^{1-i-2n} (\log R)^{j} \\
 &\quad +  \tilde\psi_2 (R^{-2}) \sum_{n=0}^{\I} \sum_{i=0}^\infty \sum_{j=0}^{\ell+i} \chi_{[-2n-i=-1]}b_{nij}\,   (\log R)^{j+1} 
 }
  which again lies in $S^{2}(R \,(\log R)^\ell )$. 
    \end{proof}

We remark that $v_{1}$ as constructed above satisfies
  \[
  v_{1} \in  \frac{\lambda^{\frac12}}{\mu(t)^{2}} S^{2}(R), 
  \]
see \eqref{fasymp}.

\medskip

{\bf Step 2:} {\em Show that if $ v_{2k-1}$ is chosen as above then
  \eqref{e2k-1} holds.}

\smallskip \noindent Thinking of $v_{2k-1}$ as a function of $t$, $R$
and $a$ we can write $e_{2k-1}$ in the form
\[
e_{2k-1} = N_{2k-1}(v_{2k-1}) + E^t v_{2k-1} + E^a v_{2k-1}
\]
Here $N_{2k-1}(v_{2k-1})$ accounts for the contribution from the
nonlinearity and is given by \eqref{eodd}. $E^t v_{2k-1}$ contains
the terms in
\begin{equation}\label{eq:prtt}
  -\partial_{tt} v_{2k-1} (t,R,a)
\end{equation}
where no derivative applies to the variable $a$, while $ E^a
v_{2k-1}$ contains those terms in
\[
\Big(\partial_{tt}-\pr_{rr} -\frac{2}{r}\pr_r\Big) v_{2k-1}(t,R,a)
\]
where at least one derivative applies to the variable~$a$ (recall
that in Step~1 the parameter $a$ was frozen). We begin with the
terms in $N_{2k-1}$. We first note that, by summing the $v_j$ over
$1\le j\le 2k-2$,
\begin{equation}\label{eq:uentw1}
  u_{2k-2} - u_0 \in \frac{\lambda^{\frac12}}{\mu^{2}} \IS^2(R \, (\log R)^{n},\mathcal
  Q)\end{equation}
  for some integer $n=n(k)\ge0$. 
The first term in $N_{2k-1}(v_{2k-1})$ contributes
\begin{eqnarray}
t^2 (u_{2k-2}^4-u_0^4)v_{2k-1}&=t^2 [(u_{2k-2}-u_0)^4 +
4(u_{2k-2}-u_0)^3u_0 \nn\\
&\quad+ 6(u_{2k-2}-u_0)^2 u_0^2 + 4 (u_{2k-2}-u_0) u_0^3]
v_{2k-1}\label{eq:udiff1}
\end{eqnarray}
Using \eqref{eq:uentw1} we compute
\begin{align*}
& t^2 (u_{2k-2}-u_0)^4 v_{2k-1} \\&\in \frac{1}{(t\lambda)^6}
\IS^8(R^4  \, (\log R)^{4n},\cQ) \frac{\lambda^{\frac12}}{(t \lambda)^{2k}}
\IS^2(R\, (\log R)^{m_{k}},\cQ) \\
&\subset a^6 \IS^2(R^{-2}\, (\log R)^{4n},\cQ) \frac{\lambda^{\frac12}}{\mu^{2k}} \IS^2(R\, (\log R)^{m_{k}},\cQ)\\
& \subset \frac{\lambda^{\frac12}}{\mu^{2k}}
\IS^2(R^{-1}\, (\log R)^{p_{k}},\cQ)
\end{align*}
as well as
\begin{align*}
& t^2 (u_{2k-2}-u_0)u_0^3 v_{2k-1} \\ &\in t^2
\frac{\lambda^{\frac12}}{(t \lambda)^{2}} \IS^2(R\, (\log R)^{n},\mathcal
  Q) \lambda^{\frac32} S^0(R^{-3})
 \frac{\lambda^{\frac12}}{(t \lambda)^{2k}}
\IS^2(R\, (\log R)^{m_{k}},\cQ)\\
&\subset \frac{\lambda^{\frac12}}{(t \lambda)^{2k}}
\IS^2(R^{-1}\, (\log R)^{p_{k}},\cQ)
\end{align*}
The other two terms in \eqref{eq:udiff1} are similar. Next, compute
\begin{align*}
  t^2 v_{2k-1}^5 &\in \frac{t^2 \lambda^{\frac52}}{(t
\lambda)^{10k}} \IS^{10}(R^5\, (\log R)^{5m_{k}},\cQ ) \\
&\subset \frac{\lambda^{\frac12}R^6}{(t \lambda)^{10k-2}}
\IS^4(R^{-1}\, (\log R)^{5m_{k}},\cQ )\\
&\subset \frac{\lambda^{\frac12}}{(t \lambda)^{2k}} {a^6} b^{8(k-1)}
\IS^2(R^{-1}\, (\log R)^{5m_{k}},\cQ ) \\&\subset \frac{\lambda^{\frac12}}{(t
\lambda)^{2k}}
 \IS^2(R^{-1}\, (\log R)^{p_{k}},\cQ )
\end{align*}
and
\begin{align*}
  & t^2 u_{2k-2}^3\,v_{2k-1}^2 \\ &\in
\lambda^{-\frac12} (t\lambda)^2 \IS^0(R^{-3}, \cQ )
\frac{\lambda}{(t\lambda)^{4k}} \IS^4(R^2\, (\log R)^{2m_{k}},\cQ ) \\
&\subset \frac{\lambda^{\frac12}}{(t\lambda)^{2k}} b^{2k-2}
\IS^4(R^{-1}\, (\log R)^{2m_{k}},\cQ )\\&\subset
\frac{\lambda^{\frac12}}{(t\lambda)^{2k}}  \IS^2(R^{-1}\, (\log R)^{p_{k}},\cQ )
\end{align*}
with similar statements for $u_{2k-2}^2 v_{2k-1}^3$ and $u_{2k-2}
v_{2k-1}^4$.
 Summing up we obtain
\[
N_{2k-1}(v_{2k-1}) \in \frac{\lambda^{\frac12}}{(t\lambda)^{2k}}
\IS^2(R^{-1}\, (\log R)^{p_{k}},\cQ ) \subset
\frac{\lambda^{\frac12}}{(t\lambda)^{2k}} \IS^2(R^{-1}\, (\log R)^{p_{k}},\cQ ')
\]
This concludes the analysis of $N_{2k-1}(v_{2k-1})$. We continue
with the terms in $E^t v_{2k-1}$, where we can neglect the $a$
dependence. Therefore, it suffices to compute
\[
t^2 \partial_t^2 \left( \frac{\lambda^{\frac12}}{(t \lambda)^{2k}}
\IS^2(R\, (\log R)^{m_{k}} )\right) \subset \frac{\lambda^{\frac12}}{(t \lambda)^{2k}}
\IS^2(R\, (\log R)^{m_{k}})
\]
Finally, we consider the terms in $E^a v_{2k-1}$. With
\[
v_{2k-1}(r,t) = \frac{\lambda^{\frac12}}{(t \lambda)^{2k}} w(R,a),
\quad w \in S^2(R\, (\log R)^{m_{k}},\mathcal Q)
\]
we have
\begin{eqnarray*}
  t^2 E^a v_{2k-1}& =&   -2t\partial_t\left(\frac{\lambda^{\frac12}}{(t \lambda)^{2k}}\right)
  aw_a(R,a) + \frac{\lambda^{\frac12}}{(t \lambda)^{2k}}
  \big[ 2(\nu+1)aR w_{aR}(R,a)   \\
  && - 2Ra^{-1} w_{Ra}  - 2a^{-1} w_a(R,a)+
  (a^2-1) w_{aa}(R,a) + 2a w_a(R,a)\big]
\end{eqnarray*}
Since $\mathcal Q$ are even in $a$ we conclude that
\[
(1-a^2) \partial_{aa},\, a \partial_a,\, a^{-1}\pr_a\,:\; \mathcal
Q \to \mathcal Q'
\]
and therefore
\[
t^2 E^a v_{2k-1} \in \frac{\lambda^{\frac12}}{(t \lambda)^{2k}}
\IS^2(R\, (\log R)^{m_{k}},\mathcal Q')
\]
This concludes the proof of \eqref{e2k-1}. 

\medskip

{\bf Step 3:} {\em Define $v_{2k}$ so that \eqref{v2k} holds.}

\smallskip\noindent   We begin the analysis with $e_{2k-1}$ replaced by
its main asymptotic component at $R = \infty$:
\EQ{\label{f2k}
t^2  f_{2k-1} = \frac{\la^{\f12}}{\mu^{2k-1}} \Big(  \sum_{j=0}^{p_{k}} aq_j(a)
(\ln R)^j  +  &b \sum_{j=0}^{p_{k}+1} [\tilde q_j^1(a) + a\tilde q_j^2(a)]
(\ln R)^j\\
&+b^2 \sum_{j=0}^{p_{k}+1} \tilde{\tilde{q}}_j(a)(\ln R)^j  \Big) , \quad q_j, \ti q_{j}^{1,2}, \tilde{\tilde{q}}_j \in \mathcal Q'
}
This is chosen such that $t^2(e_{2k-1} - f_{2k-1})$ consists of terms either decaying at least like $R^{-1}(\ln R)^{p_k+1}$, or else gaining at least a factor $b^2$. 
Consider the equation \eqref{vkeven} with
$ f_{2k-1} $ on the right-hand side,
\[
t^2\left(-\partial_t^2+\partial_r^2 +\frac{2}r \partial_r \right) \ti v_{2k} = -t^2 f_{2k-1}
\]
Homogeneity considerations suggest that we should look for a
solution $\ti v_{2k}$ which has the form
\begin{align*}
\ti v_{2k} = \frac{\la^{\f12}}{\mu^{2k-1}} \Big( \sum_{j=0}^{p_{k}} W_{2k}^j(a) 
(\ln R)^j  +  &b\sum_{\kappa = 1,2}\sum_{j=0}^{p_{k}+1} \tilde W_{2k}^{j,\kappa}(a) 
(\ln R)^j\\
&+b^2\sum_{j=0}^{p_{k}+1} \tilde{\tilde{W}}_{2k}^{j}(a) 
(\ln R)^j  \Big) 
\end{align*}
The one-dimensional equations for $ W_{2k}^j$ are obtained by
matching the powers of $\ln R$. This gives the system of equations
\begin{align*}
  t^2 \left(-\partial_t^2+\partial_r^2 +\frac{2}r \partial_r \right) \left(\frac{\la^{\f12}}{\mu^{2k-1}}
    W_{2k}^{j}(a)\right) & = \frac{\la^{\f12}}{\mu^{2k-1}} (a q_{j}(a)
  -F_j(a)) \\
  t^2 \left(-\partial_t^2+\partial_r^2 +\frac{2}r \partial_r \right) \left(\frac{\la^{\f12}}{\mu^{2k}}
    \ti W_{2k}^{j,\kappa}(a)\right) & = \frac{\la^{\f12}}{\mu^{2k}} (a^{\kappa-1}\ti q_{j}^{\kappa}(a)
  -\ti F_j^{\kappa}(a)),\,\kappa = 1,2,\\
   t^2 \left(-\partial_t^2+\partial_r^2 +\frac{2}r \partial_r \right) \left(\frac{\la^{\f12}}{\mu^{2k+1}}
    \tilde{\tilde{W}}_{2k}^{j}(a)\right) & = \frac{\la^{\f12}}{\mu^{2k+1}} (\tilde{\tilde{q}}_{j}(a)
  -\tilde{\tilde{F}}_j(a)) \\
\end{align*}
where
\begin{equation}\label{eq:Fj_def}\begin{split}
F_j(a) &= (j+1)\left[((\nu+1)\nu(2k-1)+2a^{-2}) W_{2k}^{j+1} +(a^{-1}
-
  (1+\nu) a) \partial_a W_{2k}^{j+1}\right]\\
  &\qquad  + (j+2)(j+1) ((\nu+1)^2 +
a^{-2}) W_{2k}^{j+2} \\
\ti F_j^{\kappa}(a) &= (j+1)\left[(2k(\nu+1)\nu+2a^{-2}) \ti W_{2k}^{j+1,\kappa} +(a^{-1}
-
  (1+\nu) a) \partial_a \ti W_{2k}^{j+1,\kappa}\right]\\
  &\qquad  + (j+2)(j+1) ((\nu+1)^2 +
a^{-2}) \ti W_{2k}^{j+2,\kappa} \\
\tilde{\tilde{F}}_j(a) &= (j+1)\left[((\nu+1)\nu(2k+1)+2a^{-2}) \tilde{\tilde{W}}_{2k}^{j+1} +(a^{-1}
-
  (1+\nu) a) \partial_a \tilde{\tilde{W}}_{2k}^{j+1}\right]\\
  &\qquad  + (j+2)(j+1) ((\nu+1)^2 +
a^{-2}) \tilde{\tilde{W}}_{2k}^{j+2} \\
\end{split}
\end{equation}
Here we make the convention that $W_{2k}^{j} = 0$ for $j > p_{k}$ and $\ti W_{2k}^{j,\kappa} = \tilde{\tilde{W}}_{2k}^{j} = 0$ for $j > p_{k}+1$.
Then we solve the equations in this system successively for decreasing
values of $j$ from $p_{k}$ to $0$, respectively $p_{k}+1$ to~$0$.

Conjugating out the power of $t$ we get
\begin{align*}
t^2 \Big( -\Big(\partial_t+\frac{(2k-3/2)\nu-1/2}t\Big)^2+
  \partial_r^2 + \frac{2}r \partial_r  
\Big)W_{2k}^{j}(a) &= a q_{j}(a) -F_j(a) \\
t^2 \Big( -\Big(\partial_t+\frac{(2k-1/2)\nu-1/2}t\Big)^2+
  \partial_r^2 + \frac{2}r \partial_r  
\Big)\ti W_{2k}^{j,\kappa}(a) &= a^{\kappa-1}\ti q_{j}^{\kappa}(a) - \ti F_j^\kappa(a) \\
t^2 \Big( -\Big(\partial_t+\frac{(2k+1/2)\nu-1/2}t\Big)^2+
  \partial_r^2 + \frac{2}r \partial_r  
\Big)\tilde{\tilde{W}}_{2k}^{j}(a) &= \tilde{\tilde{q}}_{j}(a) -\tilde{\tilde{F}}_j(a) \\
\end{align*}
which we rewrite as an equation in the $a$ variable,
\EQ{
  L_{\beta} W_{2k}^{j} &= a q_{j}(a) -F_j(a),\quad \beta=(2k-3/2)\nu-1/2 \\
  L_{\beta} \ti W_{2k}^{j,\kappa} &= a^{\kappa-1}\ti q_{j}^{\kappa}(a) - \ti F_j^{\kappa}(a),\quad \beta=(2k-1/2)\nu-1/2\\
  L_{\beta} \tilde{\tilde{W}}_{2k}^{j} &= \tilde{\tilde{q}}_{j}(a) -\tilde{\tilde{F}}_j(a),\quad \beta=(2k+1/2)\nu-1/2 
  \label{w2kj}
 }
  where the one-parameter family of
operators $L_{\beta}$ is defined as above, i.e., 
\begin{equation}\label{eq:Lbeta_def}
L_\beta =(1-a^2) \partial_a^2 + 2(a^{-1} + a \beta - a)
\partial_a -\beta^2 + \beta
\end{equation}
see \eqref{Lbg}. 
We claim that solving this system with $0$ Cauchy data at $a=0$ yields
solutions which satisfy
\EQ{
  W_{2k}^j  &\in a\, \cQ, \qquad j=\overline{0,p_{k}}\\
 \ti W_{2k}^{j,\kappa}  &\in a^{\kappa-1}\, \cQ, \qquad j=\overline{0,p_{k}+1},\,\kappa = 1,\,2,\\
 \tilde{\tilde{W}}_{2k}^j  &\in \cQ, \qquad j=\overline{0,p_{k}+1}\\
  \label{cqk}
  }
The choice of~$\beta$ in~\eqref{w2kj} explains the appearance of the set~\eqref{menge}. 
The fact that we need integer multiples of the exponents in~\eqref{w2kj} is a result of
the nonlinearity which requires taking powers (moreover, $\cQ$ is not an algebra otherwise). 

The claim \eqref{cqk} is established as in the computation of~$v_{2}$ above, see~\cite[Lemma 3.6]{KST}
and~\cite[Lemma 3.9]{KST0}
for details. 

As in the case of $v_{2}$, we need to make some adjustments. First, we modify $\ti v_{2k}$ to ensure an
even expansion\footnote{We use the notation $\lan R\ran = \sqrt{1+R^2}$.} around~$R=0$:
\begin{align*}
\ti v_{2k} =& \frac{\la^{\f12}}{\mu^{2k}} \Big( R^{2}\lan R\ran^{-1}\sum_{j=0}^{p_{k}} a^{-1}W_{2k}^j(a) 
(\ln R)^j  +  \sum_{j=0}^{p_{k}+1} \tilde W_{2k}^{j,1}(a) 
(\ln R)^j\\
&+R^{2}\lan R\ran^{-1} \sum_{j=0}^{p_{k}+1} a^{-1}\tilde W_{2k}^{j,2}(a)(\ln R)^j + b\sum_{j=0}^{p_{k}+1} \tilde{\tilde{W}}_{2k}^{j}(a) 
(\ln R)^j   \Big) 
\end{align*}
This is not admissible because of the singularity of $\log R$ at $R=0$. We therefore modify this expression further:
\begin{align*}
&v_{2k}\\& =  \frac{\la^{\f12}}{\mu^{2k}} \Big( R^{2}\lan R\ran^{-1}\sum_{j=0}^{p_{k}} a^{-1}W_{2k}^j(a) 
\big(\frac12\ln(1+R^2)\big)^j  +  \sum_{j=0}^{p_{k}+1} \tilde W_{2k}^{j,1}(a) 
\big(\frac12\ln(1+R^2)\big)^j\\
&+R^{2}\lan R\ran^{-1} \sum_{j=0}^{p_{k}+1} a^{-1}\tilde W_{2k}^{j,2}(a)\big(\frac12\ln(1+ R^2)\big)^j + b\sum_{j=0}^{p_{k}+1} \tilde{\tilde{W}}_{2k}^{j}(a) 
\big(\frac12\ln(1+ R^2)\big)^j   \Big) 
\end{align*}

\medskip

{\bf Step 4: } {\em For $v_{2k}$ defined as above we show that the
corresponding error $e_{2k}$ satisfies \eqref{e2k}. }

\medskip\noindent 

We can write
$e_{2k}$ in the form
\[
t^2 e_{2k} = t^2 \left(e_{2k-1} - e_{2k-1}^0 \right) + t^2 \left(
  e_{2k-1}^0- \left(-\partial_t^2+\partial_r^2 +\frac{1}r \partial_r \right) v_{2k} \right) + t^2 N_{2k}(v_{2k})
\]
where $N_{2k}$ is defined by \eqref{eeven} and
\EQ{\label{e2k-10}
t^2 e_{2k-1}^0  &= \frac{\la^{\f12}}{\mu^{2k}} \Big(  \sum_{j=0}^{p_{k}} R^{2}\lan R\ran^{-1}q_j(a)
\big(\f12\ln(1 +R^{2})\big)^j \\
&\qquad +  \sum_{j=0}^{p_{k}+1} \tilde q_j^1(a)
\big(\f12\ln(1 +R^{2})\big)^j \\
&\qquad+ b\sum_{j=0}^{p_{k}+1} R^{2}\lan R\ran^{-1}\tilde q_j^2(a)
\big(\f12\ln(1 +R^{2})\big)^j\\
&\qquad+ b\sum_{j=0}^{p_{k}+1} \tilde{\tilde{q}}_j(a)
\big(\f12\ln(1 +R^{2})\big)^j  \Big) ,
}
with $q_j, \ti q_{j}^{1,2}, \tilde{\tilde{q}}_j(a) \in \mathcal Q'$. 
We begin with the first term in $e_{2k}$, which has the form
\[
t^2(e_{2k-1} - e_{2k-1}^0) \in \frac{\la^{\f12}}{\mu^{2k}}
\big[\IS^0(R^{-1} (\ln R)^{p_{k}+2},\mathcal Q') + b^2 \IS^0(R (\ln
R)^{p_{k}},\mathcal Q')\big]
\]
The second term is admissible for \eqref{e2k}. It
remains to show that
\begin{equation}
  \IS^0(R^{-1} (\ln R)^{q_{k}},\mathcal Q') \subset \IS^0(R^{-1}
  (\ln R)^{q_{k}},\mathcal Q) + b^2 \IS^0(R (\ln R)^{q_{k}},\mathcal
  Q') 
  \label{qqp}\end{equation} 
  which can be seen as follows: 
  for $w \in \IS^0(R^{-1} (\ln
R)^{q_{k}},\mathcal Q')$ we write
\[
w = (1-a^2)w + \frac{1}{(t \lambda)^2} R^{2} w
\]
Then
\[
(1-a^2)w \in \IS^0(R^{-1} (\ln R)^{q_{k}},\mathcal Q), \qquad
\frac{1}{(t \lambda)^2} R^{2} w \in b^2 \IS^{0}(R (\ln R)^{q_{k}},\mathcal
Q')
\]
as desired. The second term in $e_{2k}$ would equal $0$ if we were
to replace $\frac12\ln(1+R^2)$ by $\ln R$ in both $e_{2k}^0$ and
$v_{2k}$; in addition, as in the case of $e_{2}$ above, error terms arise
due to the replacement of $R$ with $R^{2}\lan R\ran^{-1}$. We leave those
latter terms to the reader. To illustrate  the former,  
consider the difference which is obtained upon replacing the
derivatives of $ \frac12\ln (1+R^2)$ by derivatives of $\ln R$ in the 
expression
\[
 t^2 \left(-\partial_t^2+\partial_r^2 +\frac{1}r \partial_r
\right) \left( \frac{\la^{\f12}}{(t \lambda)^{2k-1}} \sum_{j}
  W_{2k}^j(a) \Big(\frac12\ln (1+R^2)\Big)^j \right)
\]
Computing these differences one finds that the second term in $e_{2k}$
is a sum of expressions of the form
\EQ{\nn
&\frac{\la^{\f12}}{(t \lambda)^{2k-1}} \sum_{j} \Big( \frac{
  W_{2k}^j(a)}{a^{2}} \Big[S(R^{-2}) (\ln (1+R^2))^{j-1} + S(R^{-2}) (\ln
(1+R^2))^{j-2}\Big] \\
&\qquad + \frac{\partial_a W_{2k}^j(a)}{a} S(R^{-2})
(\ln (1+R^2))^{j-1} \Big)
}
Since $W_{2k}^j$ are cubic at $0$ it follows that we can pull out an
$a$ factor and see that this part of the error is in
\[
\frac{1}{(t \lambda)^{2k}} \IS^1(R^{-1} (\ln R)^{q_{k}},\mathcal Q')
\]
which is admissible by \eqref{qqp}.
The nonlinear expression is left to the reader. 

\bigskip In summary we arrive at the following result.

\begin{prop}
For any $N\ge1$ there exists a positive integer $k$ such that 
$u_{k}$ as constructed above satisfies 
\[
\int_{r\le t} |(\cL_{quintic} \, u_{k})(r,t)|^{2}\, r^{2} dr = O(t^{N}) \quad t\to0+
\] 
Moreover, $u_{k}-  W_{\lambda(t)} \in \frac{\la^{\f12}}{\mu^{2}}O(R)$, with $R=r\la$.  
\end{prop}

The following proposition is going to be 
the key for the proof of Theorem~\ref{thm:Main}. It states that the approximate solutions of
the previous proposition can be turned into actual solutions inside the light cone. 

\begin{prop}\label{prop:key} Fix $N>1$ sufficiently large. Then there is $k\geq 1$ sufficiently large and there exists a function $\eps(t, r)$ with 
\[
\eps(t, \cdot)\in t^{N}H^{1+\frac{\nu}{2}-}_{R^2 dR},\,\eps_t(t, \cdot)\in t^{N}H^{\frac{\nu}{2}-}_{R^2 dR}
\]
such that the function 
\[
u(t, r) = u_{2k-1}(t,r) +  \eps(t, r)
\]
is a solution of \eqref{eq:foccrit} inside the light cone. 
\end{prop}
\begin{proof}[Proof of Theorem~\ref{thm:Main} assuming Proposition~\ref{prop:key}]
As in \cite{KST} one observes that $u_{k}$ can be extended outside the light cone such that 
\[
\int_{r>t}[|\nabla_x u_{k}|^2 + |\nabla_x W_{\lambda(t)}|^2 + (\partial_t W_{\lambda(t)})^2+ |\nabla_{t}u_{k}|^2]\,dx<\delta
\]
for any prescribed $\delta>0$, provided $0<t<t_0$ with $t_0 = t_0(\delta, k, \nu)$ sufficiently small. But then, picking $t_0$ sufficiently small, we can arrange that the $\dot{H}^1\times L^2$-norm of 
$\big(u(t_0, \cdot), u_t(t_0, \cdot)\big)$ is less than $\delta$ in the region $r>t_0$. Then Huyghen's principle and the small energy global regularity imply that the corresponding solution remains of class $H^{1+\frac{\nu}{2}-}$ on the exterior of the light cone. 
\end{proof}

The remainder of the paper is devoted to proving Proposition~\ref{prop:key}. The idea is that up to smoother errors, the principal source term $e_{2k-1}$ (we take $u_{2k-1}$ to be the approximate solution) may be reduced to an explicit algebraic expression. In fact, due to Huyghen's principle, we can and shall modify $e_{2k-1}$ outside the light cone to simplify the analysis a bit. We shall then see that the error term $\eps(t, r)$, which will be constructed via a suitable iteration scheme, can be written as a sum of a smoother and a ``small'' term (in the sense of amplitude). This decomposition in fact results naturally from the structure of the first iterate of the scheme used to construct $\eps(t,r)$. 

\section{Setting up the iteration scheme; formulation on the Fourier side}

\noindent Introduce the variables $\tau = \nu^{-1}t^{-\nu}$, $R = \lambda(t)r$, and write $$\tilde{\eps}(\tau, R): = R\eps(t(\tau), r(\tau, R))$$ Then we get the following\footnote{In fact, the term $\partial_\tau(\dot{\lambda}\lambda^{-1})\tilde{\eps}$ was inadvertently omitted in \cite{KST}, which has no bearing on the proof there, however.} equation for $\tilde{\eps}$: 
\begin{equation}\label{eq:Rtauwave}\begin{split}
&(\partial_{\tau} + \dot{\lambda}\lambda^{-1}R\partial_R)^2 \tilde{\eps} - (\partial_{\tau} + \dot{\lambda}\lambda^{-1}R\partial_R)\tilde{\eps} + \calL\tilde{\eps}\\
&=\lambda^{-2}(\tau)R[N_{2k-1}(\eps) + e_{2k-1}]+\partial_\tau(\dot{\lambda}\lambda^{-1})\tilde{\eps};
\end{split}\end{equation}
where the operator $\calL$ is given by 
\[
\calL = -\partial_R^2 - 5W^4(R)
\]
and we have 
\[
RN_{2k-1}(\eps) = 5(u_{2k-1}^4 - u_0^4)\tilde{\eps} + RN(u_{2k-1}, \tilde{\eps}),
\]
\[
RN(u_{2k-1}, \tilde{\eps}) = R(u_{2k-1}+\frac{\tilde{\eps}}{R})^5 - R u_{2k-1}^5 - 5u_{2k-1}^4\tilde{\eps}
\]
Introducing the operator, with $\beta_{\nu}(\tau) = \frac{\dot{\lambda}(\tau)}{\lambda(\tau)}$,
\[
\mathcal{D} = \partial_{\tau} + \beta_{\nu}(\tau)(R\partial_R - 1),
\]
we can also write the above equation as 
\begin{equation}\label{eq:Rtauwave1}
\mathcal{D}^2\tilde{\eps} + \beta_{\nu}(\tau)\mathcal{D}\tilde{\eps} + \calL\tilde{\eps} = \lambda^{-2}(\tau)\big[5(u_{2k-1}^4 - u_0^4)\tilde{\eps} + RN(u_{2k-1}, \tilde{\eps}) + R e_{2k-1}\big]
\end{equation}

In the following, we shall freely borrow the facts about the spectral theory and associated distorted Fourier transform contained in \cite{KST} as well as \cite{DoKr}. In particular, we recall that there exists a Fourier basis $\phi(R, \xi)$ and associated spectral measure $\rho(\xi)$ satisfying the asymptotic expansions and growth conditions explained in \cite[Section 4]{KST} such that 
\[
\tilde{\eps}(\tau, R) = x_d(\tau)\phi_d(R) + \int_0^\infty x(\tau, \xi)\phi(R, \xi)\rho(\xi)\,d\xi
\]
For the asymptotic behavior of $\phi(R,\xi)$ in various regimes see~\eqref{phif+} and \eqref{phiphi0}. 
Here the functions $x(\tau, \xi)$ are the (distorted) Fourier coefficients associated with $\tilde{\eps}$, and $\phi_d(R)$ is the unique ground state with associated negative eigenvalue for the operator $\calL$. We also note the important asymptotic estimates 
\begin{equation}\label{eq:rhoasympto}
\rho(\xi)\simeq \xi^{-\frac{1}{2}},\,\xi\ll 1,\,\rho(\xi)\simeq \xi^{\frac{1}{2}},\,\xi\gg1.
\end{equation}
as well as the fact that near $\xi = 0$ as well as $\xi = \infty$ the spectral measure behaves like a symbol upon differentiation. 
We shall henceforth write 
\[
\underline{x}(\tau, \xi): = \left(\begin{array}{c}x_d(\tau)\\ x(\tau, \xi)\end{array}\right),\qquad \underline{\xi} = \binom{\xi_d}{\xi}
\]
Then proceeding as in \cite{DoKr}, in particular section 3.5 in loc. cit.     
which uses a variation on the procedure in \cite{KST}, we derive the following transport equation for $x(\tau, \xi)$: 
\begin{equation}\label{eq:transport}
\big(\mathcal{D}_{\tau}^2 + \beta_{\nu}(\tau)\mathcal{D}_{\tau} + \underline{\xi}\big)\underline{x}(\tau, \xi) = \calR(\tau, \underline{x}) + f(\tau, \underline{\xi}),
\end{equation}
where we have
\begin{equation}\label{eq:Rterms}
\calR(\tau, \underline{x})(\xi) = \Big(-4\beta_{\nu}(\tau)\mathcal{K}\mathcal{D}_{\tau}\underline{x} - \beta_{\nu}^2(\tau)(\mathcal{K}^2 + [\mathcal{A}, \mathcal{K}] + \mathcal{K} +  \beta_{\nu}' \beta_{\nu}^{-2}\mathcal{K})\underline{x}\Big)(\xi)
\end{equation}
 with $\beta_{\nu}(\tau) = \frac{\dot{\lambda}(\tau)}{\lambda(\tau)}$, and 
 \begin{equation}\label{eq:fterms}
 f(\tau, \xi) = \mathcal{F}\big( \lambda^{-2}(\tau)\big[5(u_{2k-1}^4 - u_0^4)\tilde{\eps} + RN(u_{2k-1}, \tilde{\eps}) + R e_{2k-1}\big]\big)\big(\xi\big)
 \end{equation}
Also the key operator 
 \[
 \mathcal{D}_{\tau} = \partial_{\tau} + \beta_{\nu}(\tau)\mathcal{A},\quad \mathcal{A} = \left(\begin{array}{cc}0&0\\0&\mathcal{A}_c\end{array}\right)
 \]
 and we have 
 \[
 \mathcal{A}_c = -2\xi\partial_{\xi} - \Big (\frac{5}{2}  + \frac{\rho'(\xi)\xi}{\rho(\xi)} \Big)
 \]
 Finally, we observe that the ``transference operator'' $\mathcal{K}$ is given by the following type of expression 
 \begin{equation}\label{eq:Kstructure}
 \mathcal{K} = \left(\begin{array}{cc}\mathcal{K}_{dd}&\mathcal{K}_{dc}\\
 \mathcal{K}_{cd}&\mathcal{K}_{cc}\end{array}\right)
 \end{equation}
 with mapping properties specified later on. 
 \\
 
In the following, we shall take advantage of the observation that the equation
\[
\big(\mathcal{D}_{\tau}^2 + \beta_{\nu}(\tau)\mathcal{D}_{\tau} + \underline{\xi}\big)\underline{x}(\tau, \xi) = \underline{f}(\tau, \xi) = \left(\begin{array}{c}f_d(\tau)\\ f(\tau, \xi)\end{array}\right)
\]
can be solved completely explicitly; in particular, imposing vanishing boundary data at $\tau = \infty$, we obtain the following expression for the continuous part $x(\tau, \xi)$: 
\begin{equation}\label{eq:para}
x(\tau, \xi) = \xi^{-\frac{1}{2}}\int_{\tau}^\infty \frac{\lambda^{\frac{3}{2}}(\tau)}{\lambda^{\frac{3}{2}}(\sigma)}\frac{\rho^{\frac{1}{2}}(\frac{\lambda^{2}(\tau)}{\lambda^{2}(\sigma)}\xi)}{\rho^{\frac{1}{2}}(\xi)}\sin\Big[\lambda(\tau)\xi^{\frac{1}{2}}\int_{\tau}^{\sigma}\lambda^{-1}(u)\,du\Big]f\big(\sigma, \frac{\lambda^{2}(\tau)}{\lambda^{2}(\sigma)}\xi\big)\,d\sigma
\end{equation}
We shall justify this below. 
On the other hand, one immediately obtains the elementary implicit relation 
\begin{equation}\begin{split}\label{eq:x_d}
&x_d(\tau) = \int_{\tau}^\infty H_d(\tau, \sigma)\tilde{f}_d(\sigma)\,d\sigma,\; \; H_d(\tau, \sigma) = -\frac{1}{2}|\xi_d|^{-\frac{1}{2}}e^{-|\xi_d|^{\frac{1}{2}}|\tau-\sigma|}\\
&\tilde{f}_d(\sigma) = f_d(\sigma) -  \beta_{\nu}(\sigma)\partial_{\tau}x_d(\sigma)
\end{split}\end{equation}
To derive \eqref{eq:para}, define an operator
\[
(Mf)(\tau,\xi):= \la^{-\f52}(\tau)\rho(\xi)^{\f12} f(\tau,\xi)
\]
and $\calD_1 := \p_\tau - 2\beta_\nu(\tau)\xi\p_\xi$. Then 
\[
M^{-1} \calD_1 M = \calD:= \p_\tau +\beta_\nu(\tau) \mathcal{A}_c 
\]
Second, define
\[
(Sf)(\tau,\xi):= f(\tau,\la^{-2}(\tau)\xi)
\]
Then $S^{-1} \p_\tau S = \calD_1$ whence
\[
M^{-1}S^{-1} \p_\tau \; S M = \calD 
\]
It follows that 
\EQ{\label{SM}
(SM)^{-1} (\p_\tau^2 + \beta_\nu(\tau)\p_\tau + \la^{-2}(\tau)\xi) SM = \calD^2 + \beta_\nu(\tau)\calD+\xi
}
Finally, one checks that $H(\tau,\xi):=\sin(\xi^{\f12} \om(\tau))$, $\om(\tau):=\int^\tau \la^{-1}$ satisfies
\[
(\p_\tau^2 + \beta_\nu(\tau)\p_\tau + \la^{-2}(\tau)\xi) H(\tau,\xi)=0
\]
Hence,
\[
H(\xi,\sigma,\tau) = \xi^{-\f12} \sin\Big(\xi^{\f12} \int_{\tau}^\sigma \la^{-1}(u)\, du \Big),\quad \sigma\ge\tau
\]
is the fundamental solution of the operator in the parentheses on the left-hand side of~\eqref{SM}. 
In order to measure the size of $x$, we use the norms 
\[
\|x\|_{L^{2,\alpha}_{d\rho}} = \|\langle\xi\rangle^{\alpha}x\|_{L^2_{d\rho}}
\]
The following lemma shall be used throughout the sequel without further mention: 
\begin{lem}\label{lem:basicpara} Assume $N$ is sufficiently large and write 
\[
\|\underline{f}\|_{L^{2, \alpha; N}_{d\rho}}: = \sup_{\tau>\tau_0}\tau^N \big(\|f(\tau, \cdot)\|_{L^{2,\alpha}_{d\rho}} + |f_d(\tau)|\big)
\]
Then defining $x(\tau, \xi)$ via \eqref{eq:para} and $x_d(\tau)$ as the unique fixed point of \eqref{eq:x_d}, we have 
\[
\|x\|_{L^{2,\alpha+\frac{1}{2}; N-2}_{d\rho}}\lesssim \|\underline{f}\|_{L^{2, \alpha; N}_{d\rho}}
\]
In fact, for the discrete spectral part we can improve this to 
\[
|x_d(\tau)|\lesssim \tau^{-N} \|\underline{f}\|_{L^{2, \alpha; N}_{d\rho}}
\]
\end{lem}
The proof is a straightforward consequence of the fact that 
\[
\big|\xi^{-\frac{1}{2}}\sin\Big[\lambda(\tau)\xi^{\frac{1}{2}}\int_{\tau}^{\sigma}\lambda^{-1}(u)\,du\Big]\big|\lesssim \tau
\]

What we shall effectively prove now is the following 
\begin{prop}\label{prop:fixedpoint} Given $N$ sufficiently large, $0<\nu\leq \frac{1}{2}$, there is $k$ sufficiently large and there exists a function $\tilde{e}_{2k-1}\in H^{\frac{\nu}{2}-}$ such that $\tilde{e}_{2k-1}|_{r\leq t} = e_{2k-1}|_{r\leq t}$ and such that \eqref{eq:transport} admits a fixed point $x(\tau, \cdot)\in \tau^{-N} L^{2,\frac{1+\frac{\nu}{2}-}{2}}_{d\rho}$, $\mathcal{D}_{\tau}x(\tau, \cdot)\in \tau^{-N-1} L^{2,\frac{\frac{\nu}{2}-}{2}}_{d\rho}$. 
\end{prop}
\begin{proof}[Proof of Proposition~\ref{prop:key}, assuming Proposition~\ref{prop:fixedpoint}] This is now a consequence of the fact that $\mathcal{F}$ is an isometry from $H^{2\alpha}_{dR}$ to $L^{2, \alpha}_{d\rho}$.  The discrepancy between $2\alpha$ on the one hand, and $\alpha$ on the other hand is due to the frequency being $\xi^{\f12}$ rather than~$\xi$). 
\end{proof}

\begin{proof}[Proof of Proposition~\ref{prop:fixedpoint}]
We shall solve \eqref{eq:transport} via an iterative scheme, namely 
\begin{equation}\label{eq:iterstep}
\big(\mathcal{D}_{\tau}^2 + \beta_{\nu}(\tau)\mathcal{D}_{\tau} + \underline{\xi}\big)\underline{x}_j(\tau, \xi) = \calR(\tau, \underline{x}_{j-1}) + f_{j-1}(\tau, \underline{\xi}),
\end{equation}
with 
\begin{equation}\label{eq:R_j-1}
\calR(\tau, \underline{x}_{j-1}) = -4\beta_{\nu}(\tau)\mathcal{K}\mathcal{D}_{\tau}\underline{x}_{j-1} - \beta_{\nu}^2(\tau)\Big(\mathcal{K}^2 + [\mathcal{A}, \mathcal{K}] + \mathcal{K} + \frac{\beta_{\nu}'}{\beta_{\nu}^2}\mathcal{K}\Big)\underline{x}_{j-1}
\end{equation}
and 
\begin{equation}\label{eq:f_j-1}
 f_{j-1}(\tau, \xi) = \mathcal{F}\big( \lambda^{-2}(\tau)\big[5(u_{2k-1}^4 - u_0^4)\tilde{\eps}_{j-1} + RN(u_{2k-1}, \tilde{\eps}_{j-1}) + R \tilde{e}_{2k-1}\big]\big)\big(\xi\big)
\end{equation}
and we of course set 
\[
\tilde{\eps}_j(\tau,R) = x_{d,j}(\tau)\phi_d(R) + \int_0^\infty x_{j}(\tau, \xi)\phi(R, \xi)\rho(\xi)\,d\xi,\,j\geq 1,
\]
while we also set $\tilde{\eps}_0 = \underline{x}_0 = 0$. In particular, $f_0 = \mathcal{F}\big(R \tilde{e}_{2k-1}\big)$, $\calR(\tau, \underline{x}_{0}) = 0$. 
This underlines the importance of the {\it{first iterate}} $\tilde{\eps}_1$ since it will determine the smoothness of subsequent iterates due to the smoothing properties of the parametrix. We next turn to a careful analysis of the first iterate. 

\section{The first iterate} In light of the fact that $t^2 e_{2k-1}\in \frac{\lambda^{\frac{1}{2}}}{(\lambda t)^{2k}}IS^0(R^{1+}, \calQ')$ and the precise definition of this space from before, it is clear that $e_{2k-1}$ is $C^\infty$-smooth except at the light cone $r = t$. Furthermore, by subtracting functions of regularity at least $H^1_{dR}$ and choosing $k$ large enough, we may replace $e_{2k-1}$ by an expression of the form\footnote{Recall that we may modify $e_{2k-1}$ arbitrarily outside the light cone.} 
\begin{equation}\label{eq:enemy}
t^2e_{2k-1} = C(\tau)\tau^{-N-2}(\log(1-a))^i(1-a)^{\frac{\nu-1}{2}},\,C(\tau) = O(1),
\end{equation}
for a finite collection of indices $i$. More precisely, we can write 
\[
t^2e_{2k-1} = \sum_i C_i(\tau)\tau^{-N-2}(\log(1-a))^i(1-a)^{\frac{\nu-1}{2}} + E_{2k-1}
\]
with $E_{2k-1}\in \tau^{-N}H^1_{R^2 dR}$. 

The principal issue now becomes what the effect of the parametrix \eqref{eq:para} on this kind of expression is. Then the key to proving Proposition~\ref{prop:fixedpoint} is the following result. 

\begin{lem}\label{lem:key}Define $\tilde{e}_{2k-1}$ to be the function obtained from $$C\tau^{-N-2}(\log(1-a))^i(1-a)^{\frac{\nu-1}{2}}|_{r\leq t}$$ by truncation to $r\leq t$. Then the function $x(\tau, \xi)$ defined via \eqref{eq:para} with 
\[
f(\tau, \xi) = \lambda(\tau)^{-2}\mathcal{F}(R\tilde{e}_{2k-1}(\tau, \cdot))(\xi)
\]
satisfies $x = x^{(1)} + x^{(2)}$, where 
\[
x^{(1)}(\tau, \cdot)\in \tau^{-N}L^{2,\frac{3}{4}-}_{d\rho}, \;\; x^{(2)}\in \tau^{-N}L^{2,\frac{1+\frac{\nu}{2}-}{2}}_{d\rho},\;\; \frac{1}{R}\chi_{[R<\frac{\nu\tau}{2}]}\mathcal{F}^{-1}(x^{(2)})\in \tau^{-N}L^\infty_{dR}
\]
provided $k$ is sufficiently large. Finally, we also get the bound 
\[
\mathcal{D}_{\tau}x \in \tau^{-N-1}L^{2,\frac{\frac{\nu}{2}-}{2}}_{d\rho},\;\;|x_d(\tau)|\lesssim \tau^{-N-1}
\]
where $x_d$ is defined via \eqref{eq:x_d}, with $f$ replaced by $R\tilde{e}_{2k-1}$. 
\end{lem}
\begin{cor}\label{cor:firstiter} Modify $e_{2k-1}$ by restricting the terms
\[
C_i(\tau)\tau^{-N-2}(\log(1-a))^i(1-a)^{\frac{\nu-1}{2}}
\]
to $r<t$ while smoothly truncating  $E_{2k-1}$ to a dilate of the light cone. Then the same bounds as in the preceding lemma apply to $\underline{x}$, the true first iterate.
\end{cor}
\begin{proof}[Proof of Corollary~\ref{cor:firstiter}] This follows from Lemma~\ref{lem:key} as well as Lemma~\ref{lem:basicpara} applied to $E_{2k-1}$. 
\end{proof}
\begin{proof}[Proof of Lemma~\ref{lem:key}] The proof of the bounds for small $\xi$ proceeds exactly as in \cite{KST}, and so we shall now focus on $\xi\gg 1$. The expression $\mathcal{F}(R\tilde{e}_{2k-1}(\tau, \cdot))(\xi)$ is given by 
\[
 \mathcal{F}(R\tilde{e}_{2k-1}(\tau, \cdot))(\xi) = \int_0^\infty \phi(R, \xi)R\tilde{e}_{2k-1}(\tau, R)\, dR
 \]
 and here we can restrict the integration to $0\leq R\leq \nu\tau$. We shall next use the precise asymptotic expansion for $\phi(R, \xi)$. Specifically, for $R\xi^{\frac{1}{2}}\gtrsim 1$, we get 
 \EQ{\label{phif+}
 \phi(R, \xi) = a(\xi)f_{+}(R, \xi) + \overline{a(\xi)}f_-(R, \xi),\quad f_-(R, \xi) = \overline{f_+(R, \xi)}
 }
 where we have the  expansions  in Hankel functions
 \begin{align*}
 f_{+}(R, \xi) &= e^{iR\xi^{\frac{1}{2}}}\sigma(R\xi^{\frac{1}{2}}, \xi),\\
 \sigma(q, \xi) &= \sum_{j=0}^\infty q^{-j}\psi_{j}^+(R)
 \end{align*}
and $\psi_j^+$ uniformly bounded in $R$ with $\psi_j^+(R) = O(R^j)$. In particular, we get 
\[
f_+(R, \xi) = e^{iR\xi^{\frac{1}{2}}} + O(\xi^{-\frac{1}{2}})
\]
for large $\xi$ and $R\xi^{\frac{1}{2}}\gtrsim 1$. 
\\
Another key ingredient is the precise asymptotic formula for the coefficients $a(\xi)$. In fact, from \cite{DoKr}, Cor. 3.7, we have 
\[
a(\xi) = \frac{\xi^{-\frac{1}{2}}}{2i}(1+ O(\xi^{-\frac{1}{2}})),\;\; \xi\gg 1
\]
Neglecting for now all terms of the form $O(\xi^{-\frac{1}{2}})$, we are then lead to the following integral:
\begin{align*}
& \mathcal{F}(R\tilde{e}_{2k-1}(\sigma, \cdot))(\xi)\\& = \frac{\xi^{-\frac{1}{2}}}{2i}\int_0^{\nu\sigma} e^{iR\xi^{\frac{1}{2}}}R\tilde{e}_{2k-1}(\sigma, R)\,dR - \frac{\xi^{-\frac{1}{2}}}{2i}\int_0^{\nu\sigma} e^{-iR\xi^{\frac{1}{2}}}R\tilde{e}_{2k-1}(\sigma, R)\,dR\\
 &=\frac{e^{i\nu\sigma\xi^{\frac{1}{2}}}\xi^{-\frac{1}{2}}}{2i}\int_0^{\nu\sigma} e^{i(R-\nu\sigma)\xi^{\frac{1}{2}}}R\tilde{e}_{2k-1}(\sigma, R)\,dR\\& - \frac{e^{-i\nu\sigma\xi^{\frac{1}{2}}}\xi^{-\frac{1}{2}}}{2i}\int_0^{\nu\sigma} e^{-i(R-\nu\sigma)\xi^{\frac{1}{2}}}R\tilde{e}_{2k-1}(\sigma, R)\,dR\\
\end{align*}
Recalling the definition of $\tilde{e}_{2k-1}$ and replacing the outer factor $R$ by $\nu\sigma$ up to an error of type $E_{2k-1}$, the preceding is seen to be equal to a linear combination of terms of the form 
\begin{align*}
&(\log \sigma)^{N_2}\sigma^{-N_1}\cos(\nu\sigma\xi^{\frac{1}{2}})\xi^{-\frac{1}{2}}\int_0^{\nu\sigma}\frac{(\log x)^i\sin(\xi^{\frac{1}{2}} x)}{x^{\frac{1-\nu}{2}}}\,dx\\
&= (\log \sigma)^{N_2}\sigma^{-N_1}\xi^{-\frac{3}{4+}}\cos(\nu\sigma\xi^{\frac{1}{2}})C(\xi)\\
&+\sum_{0\leq j\leq i}(\log \sigma)^{N_2+j}\sigma^{-N_1-\frac{1-\nu}{2}}\xi^{-1-}\cos^2(\nu\sigma\xi^{\frac{1}{2}})C_j(\xi)+O(\sigma^{-N_1-\frac{3-\nu}{2}}\xi^{-\frac{5}{4+}}),\\& N_1\geq N+1,\\
&(\log \sigma)^{N_2}\sigma^{-N_1}\sin(\nu\sigma\xi^{\frac{1}{2}})\xi^{-\frac{1}{2}}\int_0^{\nu\sigma}\frac{(\log x)^i\cos(\xi^{\frac{1}{2}} x)}{x^{\frac{1-\nu}{2}}}\,dx\\
&=(\log \sigma)^{N_2}\sigma^{-N_1}\xi^{-\frac{3}{4+}}\sin(\nu\sigma\xi^{\frac{1}{2}})C(\xi)\\
&+\sum_{0\leq j\leq i}(\log \sigma)^{N_2+j}\sigma^{-N_1-\frac{1-\nu}{2}}\xi^{-1-}\sin^2(\nu\sigma\xi^{\frac{1}{2}})C_j(\xi)+O(\sigma^{-N_1-\frac{3-\nu}{2}}\xi^{-\frac{5}{4+}}),\\& N_1\geq N+1,
\end{align*}
for some uniformly bounded functions $C(\xi), C_j(\xi)$ with symbol behavior. Note that the errors $O(\sigma^{-N_1-\frac{3-\nu}{2}}\xi^{-\frac{5}{4+}})$ are in $\sigma^{-N-2}L^{2, \frac{1}{2}}_{d\rho}$, which corresponds to $H^1_{R^2 dR}$ as for the error terms $E_{2k-1}$ introduced further above. Now apply the parametrix \eqref{eq:para} to the first two terms in the above two expansions, multiplied by $\lambda(\sigma)^{-2}$. Observe that 
\[
\lambda(\tau)\xi^{\frac{1}{2}}\int_{\tau}^{\sigma}\lambda^{-1}(u)\,du = \nu\tau(1-(\frac{\tau}{\sigma})^{\frac{1}{\nu}})\xi^{\frac{1}{2}}
\]
and so we obtain the respective contributions
\begin{align*}
&x(\tau, \xi) =
 \xi^{-\frac{5}{4+}}\int_{\tau}^\infty \frac{\lambda^{0+}(\tau)}{\lambda^{0+}(\sigma)}C(\frac{\lambda^{2}(\tau)}{\lambda^{2}(\sigma)}\xi)\frac{\rho^{\frac{1}{2}}(\frac{\lambda^{2}(\tau)}{\lambda^{2}(\sigma)}\xi)}{\rho^{\frac{1}{2}}(\xi)}\sin\big[\nu\tau(1-(\frac{\tau}{\sigma})^{\frac{1}{\nu}})\xi^{\frac{1}{2}}\big]\\&\hspace{5cm}(\log\sigma)^{N_2}\sigma^{-N_1}\cos^\kappa\big(\nu\frac{\tau^{1+\nu^{-1}}}{\sigma^{\nu^{-1}}}\xi^{\frac{1}{2}}\big)\,d\sigma,\,\kappa = 1,2,
\end{align*}
as well as 
\begin{align*}
&x(\tau, \xi) =
 \xi^{-\frac{5}{4+}}\int_{\tau}^\infty \frac{\lambda^{0+}(\tau)}{\lambda^{0+}(\sigma)}C(\frac{\lambda^{2}(\tau)}{\lambda^{2}(\sigma)}\xi)\frac{\rho^{\frac{1}{2}}(\frac{\lambda^{2}(\tau)}{\lambda^{2}(\sigma)}\xi)}{\rho^{\frac{1}{2}}(\xi)}\sin\big[\nu\tau(1-(\frac{\tau}{\sigma})^{\frac{1}{\nu}})\xi^{\frac{1}{2}}\big]\\&\hspace{5cm}(\log\sigma)^{N_2}\sigma^{-N_1}\sin^\kappa\big(\nu\frac{\tau^{1+\nu^{-1}}}{\sigma^{\nu^{-1}}}\xi^{\frac{1}{2}}\big)\,d\sigma,\,\kappa = 1,2,
\end{align*}
where now $N_1\geq N+ 1 + 2(\frac{\nu+1}{\nu})$. 
Expanding the sine function in the integrand, we encounter in the case $\kappa = 1$ the problematic terms
\begin{align*}
\sin(\nu\tau\xi^{\frac{1}{2}})\xi^{-\frac{5}{4+}}\int_{\tau}^\infty \frac{\lambda^{0+}(\tau)}{\lambda^{0+}(\sigma)}C(\frac{\lambda^{2}(\tau)}{\lambda^{2}(\sigma)}\xi)\frac{\rho^{\frac{1}{2}}(\frac{\lambda^{2}(\tau)}{\lambda^{2}(\sigma)}\xi)}{\rho^{\frac{1}{2}}(\xi)}\frac{(\log\sigma)^{N_2}}{\sigma^{N_1}}\cos^2\big(\nu\frac{\tau^{1+\nu^{-1}}}{\sigma^{\nu^{-1}}}\xi^{\frac{1}{2}}\big)\,d\sigma
\end{align*}
\begin{align*}
\cos(\nu\tau\xi^{\frac{1}{2}})\xi^{-\frac{5}{4+}}\int_{\tau}^\infty \frac{\lambda^{0+}(\tau)}{\lambda^{0+}(\sigma)}C(\frac{\lambda^{2}(\tau)}{\lambda^{2}(\sigma)}\xi)\frac{\rho^{\frac{1}{2}}(\frac{\lambda^{2}(\tau)}{\lambda^{2}(\sigma)}\xi)}{\rho^{\frac{1}{2}}(\xi)}\frac{(\log\sigma)^{N_2}}{\sigma^{N_1}}\sin^2\big(\nu\frac{\tau^{1+\nu^{-1}}}{\sigma^{\nu^{-1}}}\xi^{\frac{1}{2}}\big)\,d\sigma
\end{align*}
where the bad\footnote{in the sense of non-oscillatory} terms $\cos^2(\nu\frac{\tau^{1+\nu^{-1}}}{\sigma^{\nu^{-1}}}\xi^{\frac{1}{2}}), \sin^2(\nu\frac{\tau^{1+\nu^{-1}}}{\sigma^{\nu^{-1}}}\xi^{\frac{1}{2}})$ may be replaced up to an oscillating term by $\frac{1}{2}$, in which case the preceding further simplifies to 
\begin{equation}\label{eq:BadGuy}
x^{(2)}(\tau, \xi) = \sin(\nu\tau\xi^{\frac{1}{2}})\xi^{-\frac{5}{4+}}\int_{\tau}^\infty \frac{\lambda^{0+}(\tau)}{\lambda^{0+}(\sigma)}\frac{(\log\sigma)^{N_2}}{\sigma^{N_1}}C(\frac{\lambda^{2}(\tau)}{\lambda^{2}(\sigma)}\xi)\frac{\rho^{\frac{1}{2}}(\frac{\lambda^{2}(\tau)}{\lambda^{2}(\sigma)}\xi)}{\rho^{\frac{1}{2}}(\xi)}\,d\sigma
\end{equation}
as well as a similar term with $ \cos(\nu\tau\xi^{\frac{1}{2}})$ in front. 
Here we have not gained the necessary decay for the Fourier coefficient $x$. On the other hand, this term has a well-defined oscillatory behavior in terms of $\xi^{\frac{1}{2}}$. 
Thus for $R<\frac{\nu\tau}{2}$, we get 
\begin{align*}
\int_0^\infty \phi(R, \xi)x^{(2)}(\tau, \xi) \rho(\xi)\,d\xi& = \int_0^\infty \chi_{[R\xi^{\frac{1}{2}}\less 1]}\phi(R, \xi)x^{(2)}(\tau, \xi) \rho(\xi)\,d\xi\\
&+ \int_0^\infty \chi_{[R\xi^{\frac{1}{2}}\gtrsim 1]}\phi(R, \xi)x^{(2)}(\tau, \xi) \rho(\xi)\,d\xi\\
\end{align*}
In the first integral we have the implicit phase functions $e^{\pm i\nu\tau\xi^{\frac{1}{2}}}$, and performing integration by parts with respect to $\xi^{\frac{1}{2}}$ leads to a gain of $\simeq \tau^{-1}\xi^{-\frac{1}{2}}$, which makes the integrand absolutely integrable with respect to $d\xi$. In the second integral one has the implicit phases $e^{i(\pm R\pm \nu\tau)\xi^{\frac{1}{2}}}$, and here we gain $\simeq (\pm R\pm \nu\tau)^{-1}\xi^{-\frac{1}{2}}$. More precisely, for the first integral one writes (see \cite{KST})
\EQ{\label{phiphi0}
\phi(R, \xi) = \phi_0(R) + R^{-1}\sum_{j=1}^\infty(R^2\xi)^j\phi_j(R^2),
}
with $\phi_j(R^2)\lesssim R^2$, and if this expression is hit by a $\partial_{\xi^{\frac{1}{2}}}$, it changes to 
\[
2R\xi^{\frac{1}{2}}\sum_{j=1}^\infty j(R^2\xi)^{j-1}\phi_j(R^2),
\]
and we have $R^2\lesssim R\xi^{-\frac{1}{2}}$ on the support of the integrand. The extra factor $R$ is used to absorb the $R^{-1}$ in 
\[
\big\|\frac{1}{R}\chi_{[R<\frac{\nu\tau}{2}]}\mathcal{F}^{-1}(x^{(2)})\big\|_{L^\infty_{dR}}\lesssim \tau^{-N_1+}
\]
The case when the operator $\partial_{\xi^{\frac{1}{2}}}$ hits the other terms in the integral are handled similarly. 

We have now reduced matters to controlling integrals of the form\footnote{These arise upon replacing $\sin^2(\nu\frac{\tau^{1+\nu^{-1}}}{\sigma^{\nu^{-1}}}\xi^{\frac{1}{2}})$ by $\sin^2(\nu\frac{\tau^{1+\nu^{-1}}}{\sigma^{\nu^{-1}}}\xi^{\frac{1}{2}})-\frac{1}{2}$ and similarly for $\cos^2(\nu\frac{\tau^{1+\nu^{-1}}}{\sigma^{\nu^{-1}}}\xi^{\frac{1}{2}})$.} 
\begin{align*}
& \f{\sin(\nu\tau\xi^{\frac{1}{2}})}{\xi^{\frac{5}{4+}} }\int_{\tau}^\infty \frac{\lambda^{0+}(\tau)}{\lambda^{0+}(\sigma)}C(\frac{\lambda^{2}(\tau)}{\lambda^{2}(\sigma)}\xi)\frac{\rho^{\frac{1}{2}}(\frac{\lambda^{2}(\tau)}{\lambda^{2}(\sigma)}\xi)}{\rho^{\frac{1}{2}}(\xi)}\Xi\big(\kappa\nu\frac{\tau^{1+\nu^{-1}}}{\sigma^{\nu^{-1}}}\xi^{\frac{1}{2}}\big)\frac{(\log\sigma)^{N_2}}{\sigma^{N_1}}\,d\sigma\\
& \f{\cos(\nu\tau\xi^{\frac{1}{2}})}{\xi^{\frac{5}{4+}}}\int_{\tau}^\infty \frac{\lambda^{0+}(\tau)}{\lambda^{0+}(\sigma)}C(\frac{\lambda^{2}(\tau)}{\lambda^{2}(\sigma)}\xi)\frac{\rho^{\frac{1}{2}}(\frac{\lambda^{2}(\tau)}{\lambda^{2}(\sigma)}\xi)}{\rho^{\frac{1}{2}}(\xi)}\Xi\big(\kappa\nu\frac{\tau^{1+\nu^{-1}}}{\sigma^{\nu^{-1}}}\xi^{\frac{1}{2}}\big)\frac{(\log\sigma)^{N_2}}{\sigma^{N_1}}\,d\sigma
\end{align*}
with $\Xi = \sin, \cos$ and $\kappa\in \{1,2,3\}$. 
Here we perform integration by parts with respect to $\sigma$, which gives the desired gain $\xi^{-\frac{1}{2}}$, placing the term into $L^{2,1-}_{d\rho}$.\\ 
To conclude the bounds for the first iterate, we still need to bound the contribution from all the errors which arose when we replaced $\phi(R, \xi)$ by $c\Im(i\xi^{-\frac{1}{2}}e^{iR\xi^{\frac{1}{2}}})$, as well as the errors of type $E_{2k-1}$. Note that the former type of error contributes to the Fourier transform of $R\tilde{e}_{2k-1}$ a term of the  schematic form 
\[
\xi^{-1}\lambda(\sigma)^{-2}\int_0^{\nu\sigma}R\tilde{e}_{2k-1}\, dR
\]
and we can use a crude $L^1_{dR}$-bound for the integrand to infer that this contribution is bounded by $\lesssim \tau^{-N}\xi^{-\frac{3}{2}}$, which places this contribution into $\tau^{-N}L^{2,\frac{3}{4}-}_{d\rho}$. The contribution of errors of type $E_{2k-1}$ was handled in the proof of Corollary~\ref{cor:firstiter}. This completes the proof of the lemma. 
\end{proof}

\section{The second iterate}

We observe that Lemma~\ref{lem:key} implies the following: the first iterate $\tilde{\eps}_1$ can be written as $\tilde{\eps}_1 = \tilde{\eps}^{(1)} + \tilde{\eps}^{(2)}$ where 
\begin{equation}\label{eq:1stiterbounds}
\begin{aligned}
\big\|\tilde{\eps}^{(1)}\big\|_{H^{\frac{3}{2}-}_{dR}} + \big\|\frac{\tilde{\eps}^{(1)}}{R}\big\|_{L_{dR}^M} &\lesssim \tau^{-N},\\
\big\|\tilde{\eps}^{(2)}\big\|_{H^{1+\frac{\nu}{2}-}} + \big\|\chi_{[R<\frac{\nu\tau}{2}]}\frac{\tilde{\eps}^{(2)}}{R}\big\|_{L^\infty_{dR}} &\lesssim \tau^{-N}
\end{aligned}
\end{equation}
where $M\geq 2$ can be chosen  arbitrarily large (with implicit constant depending on $M$). By radiality, we then obtain 
\[
\big\|\frac{\tilde{\eps}^{(2)}}{R}\big\|_{L^\infty_{dR}} + \big\|\frac{\tilde{\eps}^{(2)}}{R}\big\|_{L^M_{dR}} \lesssim \tau^{-N},\,
\]
To control the second iterate $\tilde{\eps}_2$, write 
\begin{align*}
\big(\big(\mathcal{D}_{\tau}^2 + \beta_{\nu}(\tau)\mathcal{D}_{\tau} + \underline{\xi}\big)(\underline{x}_2-\underline{x}_1)\big)(\tau, \xi) = \calR(\tau, \underline{x}_{1}) + \Delta_1f_{1}(\tau, \underline{\xi})
\end{align*}
with 
\begin{align*}
\Delta_1f_{1}(\tau, \underline{\xi}): =  \mathcal{F}\big( \lambda^{-2}(\tau)\big[5(u_{2k-1}^4 - u_0^4)\tilde{\eps}_1 + RN(u_{2k-1}, \tilde{\eps}_1)\big]\big)\big(\xi\big)
\end{align*}
Then we claim 
\begin{lem}\label{lem:eps_2} We have the estimates 
\begin{align*}
&(x_2 - x_1)(\tau, \cdot) \in \tau^{-N}L^{2,1}_{d\rho},
 |(x_2-x_1)_d(\tau)| \lesssim \tau^{-N-1}, \\
&\mathcal{D}_{\tau}(x_2-x_1)(\tau, \cdot) \in \tau^{-N-1}L^{2,\frac{1}{2}}_{d\rho},\, |\partial_{\tau}(x_2-x_1)_d(\tau)| \lesssim \tau^{-N-1}
\end{align*}
hold. 
In fact, one gains a factor $\frac{1}{N}$ in the corresponding norm bounds. 
\end{lem}
\begin{proof}[Proof of Lemma~\ref{lem:eps_2}] From Lemma~\ref{lem:basicpara}, we conclude that, with $x, f$ as in \eqref{eq:para}, we have 
\[
\sup_{\tau>0}\tau^{N}\big\|x(\tau, \cdot)\big\|_{L^{2,\alpha+\frac{1}{2}}} \lesssim \sup_{\tau>0}\tau^{N+2}\big\|f(\tau, \cdot)\big\|_{L^{2,\alpha}}
\]
\[
\sup_{\tau>0}\tau^{N}\big\|\mathcal{D}_{\tau}x(\tau, \cdot)\big\|_{L^{2,\alpha}} \lesssim \sup_{\tau>0}\tau^{N+1}\big\|f(\tau, \cdot)\big\|_{L^{2,\alpha}}
\]
Further, from \cite{KST}, \cite{DoKr} we have the operator bounds (here $\alpha$ is arbitrary)
\begin{align*}
\mathcal{K}_{cc}: L^{2,\alpha}_{d\rho}\rightarrow L^{2,\alpha+\frac{1}{2}}_{d\rho},\quad [\mathcal{K}_{cc}, \mathcal{A}]: L^{2,\alpha}_{d\rho}\rightarrow L^{2,\alpha+\frac{1}{2}}_{d\rho} \\ 
\mathcal{K}_{dc}: L^{2,\alpha}_{d\rho}\rightarrow \R,\quad \mathcal{K}_{dc}: \R\rightarrow  L^{2,\alpha}_{d\rho},\quad \mathcal{K}_{dd}: \R\rightarrow \R
\end{align*}
In fact, the operator $\mathcal{K}_{cc}$ is a smoothing operator, and corresponds to the operator $\mathcal{K}_0$ in \cite{KST} . Our strategy is to exploit the smoothing effect of the wave parametrix \eqref{eq:para}. This indeed leads to a derivative gain for all the  terms in $\calR(\tau, \underline{x}_{1})$:  assuming the that the Fourier transform $\underline{x}_1$ of the first iterate can be decomposed into two terms with bounds as in Lemma~\ref{lem:key}, we get 
\[
\mathcal{R}(\tau, \underline{x}_1)\in \tau^{-N-2}L^{2,\frac{1}{2}}_{d\rho} 
\]
Application of \eqref{eq:para} leads to expressions in $\tau^{-N}L^{2, 1}_{d\rho}$ ; moreover, applying $\mathcal{D}_{\tau}$ to these terms leads to expressions in $\tau^{-N-1}L^{2,\frac{1}{2}}_{d\rho}$, as required. The contribution to the discrete spectral part is also immediate.  
\\
We next turn to the contribution of $\Delta_1f_{1}(\tau, \underline{\xi})$. In fact, we claim that all these terms lead to a contribution in $L^{2,1}_{d\rho}$ upon application of \eqref{eq:para}. 
To see this, it suffices to check that 
\begin{align*}
\mathcal{F}\big( \lambda^{-2}(\tau)\big[5(u_{2k-1}^4 - u_0^4)\tilde{\eps}_1 + RN(u_{2k-1}, \tilde{\eps}_1)\big]\big)\big(\xi\big)\in \tau^{-N-2}L^{2,\frac{1}{2}}_{d\rho}
\end{align*}
For the first term, use the relation from \cite{KST} that 
\[
u_{2k-1} - u_0 = O(\frac{\lambda^{\frac{1}{2}}R^{1+}}{(\lambda t)^2})
\]
from which we easily infer 
\[
\big\|\mathcal{F}\big( \lambda^{-2}(\tau)\big[5(u_{2k-1}^4 - u_0^4)\tilde{\eps}_1 ]) \big\|_{L^{2,\frac{1}{2}}_{d\rho}}\lesssim (\lambda t)^{-2}\|\tilde{\eps}_1\|_{H^1_{dR}}\simeq \tau^{-2}\|\tilde{\eps}_1\|_{H^1_{dR}}
\]
which is bounded from Lemma~\ref{lem:key}. To control the source term $RN(u_{2k-1}, \tilde{\eps}_1)$, we consider the two extreme possibilities 
\[
u_{2k-1}^3\frac{\tilde{\eps}_1^2}{R},\;\;(\frac{\tilde{\eps}_1}{R})^4\tilde{\eps}_1.
\]
To bound the first term on the left, we use 
\begin{align*}
\big\|\lambda^{-2}(\tau)u_{2k-1}^3\frac{\tilde{\eps}_1^2}{R}\big\|_{H^1_{dR}}&\lesssim \|\frac{\tilde{\eps}_1}{R}\|_{L_{dR}^M}^2\|\lambda^{-2}(\tau)u_{2k-1}^3\|_{L^{2+}} \\&+ \|\tilde{\eps}_1\|_{H^{1+}_{dR}}\|\frac{\tilde{\eps}_1}{R}\|_{L_{dR}^M}\|\lambda^{-2}(\tau)u_{2k-1}^3\|_{W^{1,\infty}}
\end{align*}
while for the second term we have 
\begin{align*}
\big\|(\frac{\tilde{\eps}_1}{R})^4\tilde{\eps}_1\big\|_{H^1_{dR}}&\lesssim \|\frac{\tilde{\eps}_1}{R}\|_{L^{10}_{dR}}^5+ \|\frac{\tilde{\eps}_1}{R}\|_{L^M_{dR}}^4\|\tilde{\eps}_1\|_{H^{1+}_{dR}}
\end{align*}
and we have 
\[
 \|\frac{\tilde{\eps}_1}{R}\|_{L^{10}_{dR}}\lesssim  \|\frac{\tilde{\eps}_1}{R}\|_{L^{M}_{dR}} + \|\tilde{\eps}_1\|_{H^{1}_{dR}}
 \]
 Note that the rapid decay rate of $\tilde{\eps}_1$ gives much more than $\tau^{-N-2}$-decay. 

\section{The higher iterates}
Here we repeat the procedure of the preceding section, except that we replace $\underline{x}_2-\underline{x}_1$ by $\underline{x}_k-\underline{x}_{k-1}$, $k\geq 3$, and 
we replace one copy of $\tilde{\eps}_1$ by $\tilde{\eps}_{k-1} - \tilde{\eps}_{k-2}$ in each source term. 
Then we can literally repeat what we did before.  To be specific, we claim the following: 

\begin{lem}\label{lem:eps_k} We have the bounds
\begin{align*}
&(x_k - x_{k-1})(\tau, \cdot) \in \tau^{-N}L^{2,1}_{d\rho},
 |(x_k-x_{k-1})_d(\tau)| \lesssim \tau^{-N-1}, \\
&\mathcal{D}_{\tau}(x_k-x_{k-1})(\tau, \cdot) \in \tau^{-N-1}L^{2,\frac{1}{2}}_{d\rho},\, |\partial_{\tau}(x_k-x_{k-1})_d(\tau)| \lesssim \tau^{-N-1}
\end{align*}
In fact, one gains a factor $(\frac{1}{N})^k$ in the corresponding norm bounds. 
\end{lem}
 The proof is by induction on $k$, and is in all respects identical to the one of the preceding lemma, except that there is no one factor $\tilde{\eps}_{k-1} -\tilde{\eps}_{k-2}$ involved in the analogue of $\Delta_1f_{1}(\tau, \underline{\xi})$. The factor $(\frac{1}{N})^k$ comes from the repeated time integrations. 
\end{proof}

The fixed point of \eqref{eq:transport} is now found by  iteration, which completes the proof of Proposition~\ref{prop:key}. 
\end{proof}

\bigskip

\centerline{\scshape Joachim Krieger }
\medskip
{\footnotesize
 \centerline{B\^{a}timent des Math\'ematiques, EPFL}
\centerline{Station 8, 
CH-1015 Lausanne, 
  Switzerland}
  \centerline{\email{joachim.krieger@epfl.ch}}
} 

\medskip

\centerline{\scshape Wilhelm Schlag}
\medskip
{\footnotesize
 \centerline{Department of Mathematics, The University of Chicago}
\centerline{5734 South University Avenue, Chicago, IL 60615, U.S.A.}
\centerline{\email{schlag@math.uchicago.edu}
}
} 

\bigskip


\begin{thebibliography}{10}

\bibitem{Biz}
\newblock P.\ Bizo\'n, T.\ Chmaj, Z.\ Tabor (MR2097671)
\newblock  \emph{On blowup for semilinear wave equations with a focusing nonlinearity.} 
\newblock Nonlinearity 17 (2004), no. 6, 2187--2201.



\bibitem{DoKr} 
\newblock R.\ Donninger, J.\ Krieger
\newblock \emph{Nonscattering solutions and blow up at infinity for the critical wave equation. }
\newblock preprint, arXiv: 1201.3258v1

\bibitem{DKM1} 
\newblock T.\ Duyckaerts, C.\  Kenig, F.\ Merle  (MR2781926)
\newblock \emph{ Universality of blow-up profile for small radial type II blow-up solutions of energy-critical wave equation,} 
\newblock  J.\ Eur.\ Math.\ Soc., no. 3,   \textbf{13} (2011),  533--599.

\bibitem{DKM2}  
\newblock T.\ Duyckaerts, C.\  Kenig, F.\ Merle
\newblock \emph{Universality of the blow-up profile for small type II blow-up solutions of energy-critical wave equation: the non-radial case},
\newblock  preprint, arXiv:1003.0625, to appear in JEMS. 

\bibitem{DKM3}  
\newblock T.\ Duyckaerts, C.\  Kenig, F.\ Merle
\newblock \emph{ Profiles of bounded radial solutions of the focusing, energy-critical wave equation}, 
\newblock preprint, arXiv:1201.4986, to appear in GAFA. 

\bibitem{DKM4}  
\newblock T.\ Duyckaerts, C.\  Kenig, F.\ Merle 
\newblock \emph{Classification of radial solutions of the focusing, energy-critical wave equation}, 
\newblock preprint, arXiv:1204.0031. 




\bibitem{KST0}
\newblock J.\ Krieger,   W.\  Schlag, D.\ Tataru  (MR2494455)
\newblock \emph{  Renormalization and blow up for charge one equivariant critical wave maps }
\newblock Invent.\ Math.\ 171 (2008), no.\ 3, 543--615.




\bibitem{KST}
\newblock J.\ Krieger,   W.\  Schlag, D.\ Tataru  (MR2494455)
\newblock \emph{   Slow blow-up solutions for the $H^1(\R^3)$ critical focusing semilinear wave equation.}
\newblock Duke Math.\ J., no.~1, \textbf{ 147}  (2009),   1--53.

\end{thebibliography}
\end{document}